\documentclass[10pt,leqno, final]{amsart}
\usepackage[colorlinks,linkcolor=blue,citecolor=blue,urlcolor=blue]{hyperref}
\usepackage{amsfonts,amssymb}
\usepackage{color}
\usepackage{graphicx}
\usepackage[notref,notcite]{showkeys}
\usepackage{mathtools}
\usepackage{extarrows}
\usepackage{algorithmic}   
\usepackage{caption}
\usepackage{subcaption}
\usepackage{tikz}

\allowdisplaybreaks

\usepackage{todonotes}

\newtheorem{theorem}{Theorem}[section]
\newtheorem{lemma}[theorem]{Lemma}
\newtheorem{corollary}[theorem]{Corollary}

\theoremstyle{definition}
\newtheorem{definition}[theorem]{Definition}

\newtheorem*{example*}{Example}
\newtheorem{remark}[theorem]{Remark}

\numberwithin{equation}{section}

\newcommand{\EE}{\mathcal{E}}
\newcommand{\FF}{\mathcal{F}}

\newcommand{\II}{\mathcal{I}}

\newcommand{\LL}{\mathcal{L}}

\newcommand{\RR}{\mathcal{R}}

\newcommand{\N}{\mathbb{N}}

\newcommand{\R}{\mathbb{R}}

\newcommand{\frakM}{\mathfrak{M}}

\RequirePackage{stmaryrd} % for trinorm

\newcommand{\embed}{\hookrightarrow}

\newcommand{\weakly}{\rightharpoonup}
\newcommand{\weak}{\rightharpoonup}

\renewcommand{\d}{\mathrm{d}}

\newcommand{\dual}[2]{\langle #1 , #2 \rangle}

\def\clap#1{\hbox to 0pt{\hss#1\hss}}

\newcommand{\Diss}{\operatorname{Diss}}
\newcommand{\abst}{\operatorname{dist}}
\newcommand{\BV}{\mathrm{BV}}
\begin{document}

\title[On a lack of stability of parametrized BV solutions]{On a lack of stability of parametrized BV solutions to rate-independent systems with non-convex energies and discontinuous
 loads}\thanks{This research was supported by the German Research Foundation (DFG) under grant 
number~ME~3281/9-1 within the priority program Non-smooth and Complementarity-based
Distributed Parameter Systems: Simulation and Hierarchical Optimization (SPP~1962).}

\author{Merlin Andreia} \address{Technische Universit\"at Dortmund, Fakult\"at f\"ur
  Mathematik, Lehrstuhl LSX, Vogelpothsweg 87, 44227 Dortmund, Germany}
\email{merlin.andreia@tu-dortmund.de}

\author{Christian Meyer} \address{Technische Universit\"at Dortmund, Fakult\"at f\"ur
  Mathematik, Lehrstuhl LSX, Vogelpothsweg 87, 44227 Dortmund, Germany}
\email{christian2.meyer@tu-dortmund.de}

% General info
\date{\today} 

\begin{abstract} 
   We consider a rate-independent system with nonconvex energy under discontinuous external loading. The underlying space is finite dimensional 
   and the loads are functions in $BV([0,T];\R^d)$. We investigate the stability of various solution concepts w.r.t.\ a sequence of loads 
   converging weakly$*$ in $BV([0,T];\R^d)$ with a particular emphasis on the so-called normalized, $\mathfrak{p}$-parametrized 
   balanced viscosity solutions. By means of two counterexamples, it is shown that common solution concepts are not stable 
   w.r.t.\ weak$*$ convergence of loads in the sense that a limit of a sequence of solutions associated with these loads need not be a solution 
   corresponding to the load in the limit. We moreover introduce a new solution concept, which is stable in this sense, but our examples show 
   that this concept necessarily allows ``solutions'' that are physically meaningless.
\end{abstract}

\maketitle

%%%%%%%%%%%%%%%%%%%%%%%%%%%%%%%%%%%%%%%%%%%%%%%%%%%%%%%%%%%%%%%%
%%%%%%%%%%%%%%%%%%%%%%%%%%%%%%%%%%%%%%%%%%%%%%%%%%%%%%%%%%%%%%%%
%%%%%%%%%%%%%%%%%%%%%%%%%%%%%%%%%%%%%%%%%%%%%%%%%%%%%%%%%%%%%%%%

\section{Introduction}

This paper is concerned with a rate-independent system of the form 
\begin{equation}\label{eq:ris}
    0 \in \partial \RR(\dot z(t)) + D_z \II(t, z(t)) \quad \text{f.a.a.\ } t \in (0,T), 
    \quad z(0) = z_0,
\end{equation}
where $\RR: \R^d \to \R$ is convex and positive 1-homogeneous and $\II$ denotes an energy given by 
\begin{equation*}
    \II(t, z) := \frac{1}{2}\,\dual{A z}{z} + \FF(z) - \dual{\ell(t)}{z}.
\end{equation*}
Herein $A \in \R^{d\times d}$ is symmetric and positive definite and $\FF: \R^d \to \R$ is smooth, but potentially non-convex. 
Moreover, the external load $\ell: [0,T] \to \R^d$ is a given function of bounded variation. The precise assumptions on the data 
are specified in Section~\ref{sec:assu} below.

It is well known that, even for smooth external loads, one cannot expect the existence of a weakly differentiable solution such that 
\eqref{eq:ris} is satisfied for almost all $t\in (0,T)$. We refer to the counterexample in \cite[Section~2.3]{sthomas}. 
For that reason, several alternative solution concepts have been developed, among them \emph{local solutions}, \emph{global energetic
solutions}, and \emph{(parametrized) balanced viscosity (BV) solutions}. For a comprehensive overview, we refer to \cite{MRRIS}. 
Originally, all these concepts have been introduced for smooth external loads, but, recently, several authors came up with
extensions of the classical solution concepts adapted to loads in $BV([0,T]; \R^d)$. Concerning parametrized BV solutions, we refer to \cite{KneesZanini}.
With regard to other solutions concepts, rate-independent systems with discontinuous data are 
discussed e.g.\ in \cite{Krejci2009, Rec11, Rec16, RS18} for the case of strictly convex energies.  
A comparison of various solutions concepts in case of discontinuous data is presented in \cite{KR14}.
The motivation for considering loads in $BV([0,T];\R^d)$ is manifold, reaching from applications, where loads may be switched on or off, to 
the viscous approximation of optimal control problems governed by \eqref{eq:ris}. Concerning the latter, the so-called reverse approximation property 
plays an essential role and has so far only been verified under very restrictive assumptions, see \cite[Section~6]{KMS22}. 
If one aims to avoid these assumptions, the use of loads in $BV([0,T];\R^d)$ seems to be indispensable, cf.\ \cite[Remark~6.4]{KMS22}.

Within this contribution, we investigate the stability of solutions to \eqref{eq:ris} w.r.t.\ ``natural'' notions of convergence of loads in 
$BV([0,T];\R^d)$. In generic situations, one cannot expect a sequence of loads in $BV([0,T];\R^d)$ to converge strongly in $BV([0,T];\R^d)$. 
For instance, if one discretizes a given load in $BV([0,T];\R^d)$ by means of a sequence of piecewise constant functions, then this sequence 
will in general only converge weakly$*$ in $BV([0,T];\R^d)$ as the mesh size tends to zero (unless the jump set of the load is 
exactly covered by the mesh). And even, if one considers the smoothing of a load in 
$BV([0,T;\R^d)$ by convolution with the standard mollifier, one will in general only obtain intermediate convergence of the sequence of 
mollified loads, when the regularization parameter is driven to zero. Consequently, ``generic'' notions of convergence in $BV([0,T];\R^d)$ 
are weak$*$ convergence or, at best, intermediate convergence. It is therefore reasonable to investigate if established solutions concepts 
are stable w.r.t.\ these types of convergence. The crucial question in this context is the following:\\
\emph{If one considers a sequence of loads converging 
weakly$*$ (or in the intermediate topology) and a sequence of associated solutions according to one of the established solution concepts, 
which converges to a limit function, is this limit still a solution in the sense of the respective concept?}\\
We will show by means of two examples that the answer to this question is in general negative. To be more precise, 
the notion of normalized $\mathfrak{p}$-parametrized BV solutions is even not stable w.r.t.\ intermediate convergence of the loads, whereas 
the notion of local solutions is not stable w.r.t.\ weak$*$ convergence. We therefore propose a new ``relaxed'' solution concept that 
is stable w.r.t.\ weak$*$ convergence of loads in $BV([0,T];\R^d)$. However, as especially the second example in Section~\ref{sec:extwo} 
shows, this new solution concept admits functions as solutions, which are entirely meaningless from a physical point of view.
It is therefore doubtful if it makes sense to allow for discontinuous loads in $BV([0,T];\R^d)$ in the context of rate-independent systems of the 
form \eqref{eq:ris}.

The paper is organized as follows: 
The rest of this introductory section is dedicated to the notation and our standing assumptions on the data in \eqref{eq:ris}. 
We then recall the concept of normalized $\mathfrak{p}$-parametrized BV solutions in case of loads in $BV([0,T];\R^d)$ in 
Section~\ref{sec:kneeszanini}. Thereafter, in Section~\ref{sec:acounterexample}, a first example is presented, which shows that 
this solution concept is not stable w.r.t.\ intermediate convergence of the loads. 
Section~\ref{sec:relax} is then devoted to our new relaxed solution concept and its stability w.r.t.\ weak$*$ convergence of the loads.
In Section~\ref{sec:plausibility} we investigate this new solution concept in more details, especially its relation to other solution concepts. 
We prove that, under natural assumption, every local solution is also a relaxed solution, which demonstrates that the new concept is 
rather weak. This is also underlined by the second example in Section~\ref{sec:extwo}, which shows that the new solution concepts 
provides completely unphysical ``solutions''. The paper ends with a short conclusion and an appendix on auxiliary technical results.

\subsection{Notation}
Throughout the paper $\langle \cdot,\cdot \rangle$ denotes the standard scalar product in $\R^d$ and the induced euclidean norm is given by $\|\cdot\|=\sqrt{ \langle \cdot,\cdot\rangle}$. $C$ is a generic constant larger zero and for a function $f:[a,b]\to\R^d$ the total variation is defined by 
\begin{align*}
	\mbox{Var}(f;[a,b])=\sup_{\text{\tiny partitions $\{t_k\}$ of $[a,b]$}}\sum_{k}\|f(t_k)-f(t_{k-1})\|,  
\end{align*}
where a partition $\{t_k\}$ is a finite subset of $[a,b]$ with $a=t_0<t_1<\dots<t_{n-1}<t_n=b,~n\in\N$. 

For a left (resp.\ right) continuous function $f: \R\to \R^d$ we denote the one sided limits by 
$f(s-) := \lim_{\sigma \nearrow s} f(\sigma)$ and $f(s+) := \lim_{\sigma \searrow s} f(\sigma)$, respectively.

For convenience of the reader, let us collect some well known facts on functions of bounded variation that will be useful throughout the paper. 
For the proofs, we refer to \cite{Attouch, Natanson}.
By $BV([0,T];\R^d)$ we denote the space of functions with bounded total variation, i.e., 
\begin{equation*}
    BV([0,T];\R^d) := \{ f: [0,T] \to \R^d \, : \, \mbox{Var}(f;[0,T]) < \infty \}.
\end{equation*}
It is well known that a function in $BV([0,T];\R^d)$ is measurable and bounded and admits at most a countable number of discontinuities.
Thus, a function in $BV([0,T];\R^d)$ is absolutely integrable and we can equip $BV([0,T];\R^d)$ with the norm
$\|f\|_{BV([0,T];\R^d)} := \|f\|_{L^1(0,T;\R^d)} + \mbox{Var}(f;[0,T])$.
Moreover, we introduce the space of (equivalence classes of) functions of bounded variation on $(0,T)$ as
\begin{equation*}
    \BV(0,T;\R^d) := \{ f \in L^1(0,T;\R^d) : \, D f \in \mathfrak{M}(0,T;\R^d) \},
\end{equation*}
where $D f$ denotes the distributional derivative and $\mathfrak{M}(0,T;\R^d)$ is the space of $\R^d$-valued regular Borel measures on $(0,T)$.
We equip $\BV(0,T;\R^d)$ with the norm $\|f\|_{\BV(0,T;\R^d)} := \|f\|_{L^1(0,T;\R^d)} + |D f|(0,T)$, 
where $|Df|(0,T)$ denotes the total variation of the measure $Df$. It is well known that 
\begin{equation}\label{eq:equivvar}
    |Df|(0,T) = \inf \{ \mbox{Var}(v;[0,T]) : \, \text{$v$ is a representative of $[f]$} \} .
\end{equation}
Furthermore, $\BV(0,T;\R^d)$ embeds continuously in $L^\infty(0,T;\R^d)$ and compactly in every $L^p(0,T;\R^d)$, $p < \infty$.
Moreover, every function $f\in BV([0,T];\R^d)$ is a representative of an element in $\BV(0,T;\R^d)$ and,
in view of \eqref{eq:equivvar}, it holds 
\begin{equation}\label{eq:BVnormest}
    \|f\|_{BV([0,T];\R^d)} \geq \|f\|_{\BV(0,T;\R^d)}.
\end{equation}
Here and in the following, with a little abuse of notation, 
we denote the function in $BV([0,T];\R^d)$ and the associated equivalence class in $\BV(0,T;\R^d)$ by the same symbol.
Note that \eqref{eq:BVnormest} is only satisfied with equality by particular representatives of $f$, e.g., by the left or right continuous
representative. As usual, we call a sequence $(f_n)_{n\in \N} \subset \BV(0,T;\R^d)$ weakly$*$ converging in $\BV(0,T;\R^d)$, iff 
\begin{equation*}
    f_n \weak^* f \quad\text{in } \BV(0,T;\R^d) 
    \quad :\Longleftrightarrow \quad 
    \left\{\;
    \begin{aligned}
        f_n & \to f & &\text{in } L^1(0,T;\R^d),\\
        D f_n & \weak^* D f & &\text{in } \frakM(0,T;\R^d). 
    \end{aligned}
    \right.        
\end{equation*}
Every bounded sequence in $\BV(0,T;\R^d)$ contains a weakly$*$ converging subsequence and, by \eqref{eq:BVnormest}, 
the same holds true for a bounded sequence in $BV([0,T];\R^d)$ (and the associated sequence of equivalence classes, respectively).
Furthermore, by Helly's selection principle, every bounded sequence in $BV([0,T];\R^d)$ admits a subsequence that converges pointwise 
everywhere in $[0,T]$ and, if the subsequence converges weakly$*$ in $\BV(0,T;\R^d)$, too, 
the pointwise limit is a representative of the weak$*$ limit.
It is to be noted however that the pointwise limit need not be the representative attaining the infimum in \eqref{eq:equivvar}, even 
if this is true along the converging subsequence.
Finally, we say that a sequence $(f_n)_{n\in \N}$ converges w.r.t.\ intermediate convergence in $\BV(0,T;\R^d)$ and $BV([0,T];\R^d)$, 
respectively, if it converges weakly$*$ and, in addition, $|D f_n|(0,T) \to |D f|(0,T)$ as $n\to \infty$.

\subsection{Standing Assumptions}\label{sec:assu}

\subsubsection*{Energy}
Throughout the paper, $A\in\R^{d\times d}$ is symmetric and positively definite. Furthermore, $\FF:\R^d\to\R$ satisfies
\begin{align}
	&\quad\FF\in C^2(\R^d;\R),~\FF\geq 0 \label{eq:F1}\\
	&\|D^2\FF(z)v\|\leq C(1+\|z\|^q)\|v\|~\forall v,z\in\R^d \label{eq:F2} 
\end{align}  
for some $q\geq1$. 
For a given external load $\ell\in BV([0,T];\R^d)$, the energy functional $\II:[0,T]\times \R^d\to\R$ reads 
\begin{equation*}
    \II(t,z) := \frac{1}{2}\,\langle Az,z\rangle +\FF(z)-\langle\ell(t),z\rangle .
\end{equation*}
By denoting the time independent part as 
\begin{equation}\label{eq:defEE}
	\EE(z):= \frac{1}{2}\,\langle Az,z\rangle +\FF(z),
\end{equation}
the energy can be written as $\II(t,z)=\EE(z)-\langle\ell(t),z\rangle$. 
Below we will deal with other loads, denoted by $\hat{\ell},\hat{\ell}_n\in BV([0,S];\R^d)$, too, and so, for convenience, we define 
\begin{align}
	\hat{\II}(s,z)&=\EE(z)-\langle \hat{\ell}(s),z\rangle, \label{eq:defhatI}\\
	\hat{\II}_n(s,z)&=\EE(z)-\langle \hat{\ell}_n(s),z\rangle.
\end{align}

\subsubsection*{Dissipation}
For the dissipation $\RR:\R^d\to [0,\infty)$, we assume
\begin{align}
	&\RR \mbox{~is~proper,~convex,~and~lower~semicontinuous}, \label{R1}\\
	&\RR \text{ is positive 1-homogeneous, i.e., }  
	\RR(\lambda z)=\lambda\RR(z) ~\forall \,z\in\R^d,~\lambda>0, \label{R2}\\
 	&\exists \,c, C >0, \mbox{~such~that}~c\,\|z\|\leq\RR(z)\leq C\,\|z\|~\forall \,z\in\R^d.  \label{R3}
\end{align}

\subsubsection*{Initial data}
There exists $\ell_0\in\R^d$ such that $-D_z\EE(z_0)+\ell_0\in\partial\RR(0)$.

\section{A solution concept for discontinuous loads in $BV$}\label{sec:kneeszanini}

There exists a variety of solution concepts for \eqref{eq:ris}. Here we only present a brief overview with a special emphasis on 
parametrized balanced viscosity solutions and refer to \cite{MRRIS} and the references therein for more details. 
The most natural notion of solutions is the \emph{differential solution}, where one searches 
for a weakly differentiable function $z\in W^{1,1}(0,T;\R^d)$, which satisfies \eqref{eq:ris} almost everywhere. 
The disadvantage of this concept is that one cannot guarantee the existence of such a solution in case of a non-convex energy as the 
example in \cite[Section~2.3]{sthomas} shows. For that reason, several alternative solutions concepts have been developed.
The most prominent one is probably the concept of \emph{global energetic solutions}, where the subdifferential inclusion in 
\eqref{eq:ris} is replaced by a global stability condition and an energy balance, which benefits in solvability under milder assumptions, 
see e.g.\ \cite{MT99, MTL02, MT04}. 
One of the weakest solution concepts providing only a minimum of information
is the so called \emph{local solution} concept, where a solution $z$ only satisfies a local stability condition together with an energy inequality.
We will come back to this concept in Definition~\ref{def:locsol} below.
Another approach providing more information especially about the discontinuities of a solution
is given by the concept of \emph{parametrized balanced viscosity (BV) solutions}, 
where a solution consists of the tuple $(z,t)$, representing the state and the physical time,
respectively, as functions of a curve parameter $s$. To the best of our knowledge, it was first introduced in \cite{EM06}, 
but has by now been analyzed by 
various authors in multiple aspects, we only refer to \cite{MRS12} and the reference therein.       

Initially, all these concepts have been developed for smooth external loads, but recently some of them have been transferred 
to the case of discontinuous loads in $BV([0,T])$. Concerning global energetic solutions, we exemplarily refer to \cite{Krejci2009}.
In \cite{KneesZanini} the concept of parametrized BV solutions concept has been transferred to loads in $BV([0,T])$. 
In our setting, where the underlying space is finite dimensional, it reads as follows:

\begin{definition}\label{def:paramsol}
    Let $\ell \in BV([0,T];\R^d)$ be given. 
	A triple $(S,\hat{t},\hat{z},\hat{\ell})\in(0,\infty)\times W^{1,\infty}(0,S)\times W^{1,\infty}(0,S;\R^d)\times BV([0,S];\R^d)$ 
	is called \emph{normalized, $\mathfrak{p}$-parametrized balanced viscosity (BV) solution} of the 
	rate-independent system \eqref{eq:ris} associated with $\ell$, if 
    \begin{itemize}
    	\item the \emph{initial and end time condition} hold
    \begin{equation}
       \hat{t}(0)=0,~ \hat{t}(S)=T,~ \hat{z}(0)=z_0, \label{eq:initendcond}
    \end{equation}
    \item the \emph{complementarity relations and normalization condition} are fulfilled for almost all $s\in(0,S)$
    \begin{gather}
        \hat{t}'(s)\geq 0,~\hat{t}'(s)\abst(-D_z\hat{\II}(s,\hat{z}(s)),\partial\RR(0))=0, \label{eq:complementary}\\
		\hat{t}'(s)+\RR(\hat{z}'(s))+\|\hat{z}'(s)\|\abst(-D_z\hat{\II}(s,\hat{z}(s)),\partial\RR(0))=1, \label{eq:normalization}
    \end{gather}
	\item the \emph{energy identity} is valid for all $s_1, s_2\in[0,S]$
    \begin{equation}\label{eq:energy}
	\begin{aligned}
		\EE(\hat{z}(s_2))+\int_{s_1}^{s_2}\RR(\hat{z}'(r))+\|\hat{z}'(r)\|\abst(-D_z\hat{\II}(r,\hat{z}(r)),\partial\RR(0)) dr \qquad\qquad & \\[-1ex]
		=\EE(\hat z(s_1))+\int_{s_1}^{s_2}\langle \hat{\ell}(r),\hat{z}'(r)\rangle dr,
	\end{aligned}    
    \end{equation}
	\item the parametrized load $\hat{\ell}$ is compatible with $\ell$ in the following sense: for every $t_*\in[0,T]$, 
	there exists $s_*\in\hat{t}^{-1}(t_*)$ such that for all $s\in \hat{t}^{-1}(t_*)$
	\begin{equation}
		\hat{\ell}(s)=\begin{cases}	\ell( t_*-), &s<s_*,\\
							\ell( t_*+), &s>s_*,
				  \end{cases}\qquad
				  \mbox{ and } \qquad
				  \hat{\ell}(s_*)\in \{\ell(t_*),\ell(t_*-),\ell(t_*+)\}.\label{eq:defell}
	\end{equation}
 \end{itemize}
 Here and in the following, we set $0- := 0$ and $T+ := T$.
\end{definition}

\begin{remark}
    The above definition differs from the ``classical'' notion of normalized, $\mathfrak{p}$-parametrized BV solutions 
    for smooth external loads according to, e.g., \cite{MRS12}  only in the additional 
    compatibility condition in \eqref{eq:defell}. This condition basically says that, if $\ell$ is discontinuous in a viscous jump, then it only attains 
    its left and right limits during the viscous jump. Note that, if $\ell$ is continuous in $t^*$, 
    then $\hat \ell(s) = \ell(\hat t (s))$ for all $s\in \hat t^{-1}(t^*)$.
\end{remark}

A proof concerning the existence of such solutions can be found in \cite[Prop. 4.2]{KneesZanini}. The authors used a vanishing viscosity approach and showed that a parametrized version of
the solutions solving the regularized problems, which are 
\begin{align*}
	0\in\partial\RR(\dot{z}_\epsilon(t))+\epsilon\dot{z}_\epsilon(t)+D_z\II(t,z_\epsilon(t)),\quad z(0)=z_0,\quad \epsilon>0,
\end{align*}
convergence to a $\mathfrak{p}$-parametrized BV solution for passing the viscosity parameter $\epsilon\to 0$. This parametrization is done by means of the vanishing viscosity contact 
potential 
\begin{equation*}
	\mathfrak{p}:\R^d\times\R^d\to\R,\quad \mathfrak{p}(v,w)=\RR(v)+\|v\|\abst(w,\partial\RR(0)),
\end{equation*}
which also motivates the name of the solution concept. To be more precise, the parametrization is given by
\begin{equation*}
	s_\epsilon(t)=t+\int_0^t \mathfrak{p}(\dot{z}_\epsilon(r),-D_z\II(r,\dot{z}_\epsilon(r)) dr, \quad \hat{z}_\epsilon=z_\epsilon(\hat{t}_\epsilon(s)),\quad 
	\hat{\ell}_\epsilon(s)=\ell(\hat{t}_\epsilon(s)),
\end{equation*}
where $\hat{t}_\epsilon:[0,S_\epsilon]\to[0,T]$ with $S_\epsilon=s_\epsilon(T)$ is the inverse function of $s_\epsilon$, so that $\mathfrak{p}$ appears in the normalization condition 
\eqref{eq:normalization}, i.e. $\hat{t}(s)+\mathfrak{p}(\hat{z}'(s),-D_z\hat{\II}(s,\hat{z}(s)))=1$ for almost all $s\in(0,S)$.

\subsection{A first counterexample}\label{sec:acounterexample}

In this section we discuss an example, which demonstrates that the solution concept from Definition \ref{def:paramsol} 
is not stable w.r.t.\ intermediate convergence of a sequence of loads $(\ell_n)_{n\in \N}$ in $BV([0,T])$ in the following sense: 
we construct a sequence of solutions $(S_n,\hat{t}_n,\hat{z}_n,\hat{\ell}_n)_{n\in\N}$ associated with $\ell_n$ in the sense of 
Definition~\ref{def:paramsol}, which converges (weakly) to a limit, but this limit is no solution associated with the limit of $(\ell_n)_{n\in\N}$ w.r.t. intermediate convergence.
Our example is one-dimensional and we consider the following dissipation and energy:
\begin{align}\label{eq:dissandenergy}
	 \RR(z)= |z|,\qquad \EE(z) = \frac{1}{2}z^2-z,\qquad \II(t,z)=\EE(z)-\ell(t)z,
\end{align}
where $\ell\in BV([0,2])$. The initial value and end time are set to $z_0=0$ and $T=2$. The sequence of external loads given by
\begin{equation*}
	\ell_n(t)=\begin{cases}
				0,& t\in[0,1]\\
				\frac{n}{2}t-\frac{n}{2}, &t\in(1,1+\frac{1}{n})\\
				\frac{1}{2}, &t\in[1+\frac{1}{n},2].
			\end{cases}
\end{equation*}
Then $\ell_n$ converges to 
\begin{align}
	\ell=\begin{cases} 0, &t\in [0,1]\\ \frac{1}{2}, &t\in (1,2]\end{cases}\label{eq:limell}
\end{align}
 w.r.t. intermediate convergence in $BV([0,T])$.
By direct calculations one verifies that a normalized, $\mathfrak{p}$-parameterized BV solution associated with $\ell_n$ is given by
\begin{equation}\label{eq:hatzex}
    \hat{z}_n(s)=
    \begin{cases}
        0, &s\in[0,1],\\
        \frac{\frac{n}{2}}{1+\frac{n}{2}}s-\frac{\frac{n}{2}}{1+\frac{n}{2}}, &s\in (1,\frac{3}{2}+\frac{1}{n}),\\    
        \frac{1}{2}, &s\in[\frac{3}{2}+\frac{1}{n},\frac{5}{2}],
    \end{cases}
\end{equation}
\begin{equation}\label{eq:hattex}
    \hat{t}_n(s)=
    \begin{cases}
        s, &s\in[0,1],\\
        \frac{1}{1+\frac{n}{2}}s+\frac{\frac{n}{2}}{1+\frac{n}{2}}, &s\in (1,\frac{3}{2}+\frac{1}{n}),\\
        s-\frac{1}{2}, &s\in[\frac{3}{2}+\frac{1}{n},\frac{5}{2}],
    \end{cases}
\end{equation}
along with $S_n=\frac{5}{2}$ and
\begin{equation}\label{eq:hatellex}
    \hat{\ell}_n(s)=\ell_n(\hat{t}_n(s))=
    \begin{cases}
        0, &s\in[0,1],\\
        \frac{\frac{n}{2}}{1+\frac{n}{2}}s-\frac{\frac{n}{2}}{1+\frac{n}{2}}, &s\in (1,\frac{3}{2}+\frac{1}{n}),\\
        \frac{1}{2}, &s\in[\frac{3}{2}+\frac{1}{n},\frac{5}{2}].
    \end{cases}
\end{equation}
The pointwise-a.e.\ limits of these sequences for $n\to \infty$ read
    \begin{equation}\label{eq:limitzt}
		\hat{z}(s)=\begin{cases}
					0, &s\in[0,1],\\
					s-1, &s\in (1,\frac{3}{2}),\\
					\frac{1}{2}, &s\in[\frac{3}{2},\frac{5}{2}],
					\end{cases}\quad
		\hat{t}(s)=\begin{cases}
					s, &s\in[0,1],\\
					1, &s\in (1,\frac{3}{2}),\\
					s-\frac{1}{2}, &s\in[\frac{3}{2},\frac{5}{2}],
					\end{cases}
    \end{equation}
\begin{equation}\label{eq:limitell}
		\hat{\ell}(s)=\begin{cases}
					0, &s\in[0,1],\\
					s-1, &s\in (1,\frac{3}{2}),\\
					\frac{1}{2}, &s\in[\frac{3}{2},\frac{5}{2}].
					\end{cases}    
\end{equation}					
Note that $\hat z$ and $\hat t$ are not only the pointwise limits of $\hat z_n$ and $\hat t_n$, but also the strong limits in $H^1(0,S)$ and 
weak* limits in $W^{1,\infty}(0,S)$. Moreover, $\hat\ell$ is the limit of $\hat\ell_n$ w.r.t.\ intermediate convergence in $BV([0,S])$.
However, $(S,\hat{t},\hat{z},\hat{\ell})$ is no normalized, $\mathfrak{p}$-parametrized BV solution associated to the limit $\ell$ from \eqref{eq:limell} 
according to Definition \ref{def:paramsol}, since condition \eqref{eq:defell} is violated in the viscous jump $(1,\frac{3}{2})$.
Here $\hat{\ell}$ cannot be expressed by the left or right 
hand side limit of $\ell$, i.e., $\hat{\ell}(s)\notin\{\ell(\hat{t}(s)-),\ell(\hat{t}(s)),\ell(\hat{t}(s)+)\}$ for all $s\in(1,\frac{3}{2})$.

 %%%%%%%%%%%%%%%%%%%%%%%%%%%%%%%%%%%%%%%%%%%%%%%%%%%%%%%%%
 %%%%%%%%%%%%%%%%%%%%%%%%%%%%%%%%%%%%%%%%%%%%%%%%%%%%%%%%%
 %%%%%%%%%%%%%%%%%%%%%%%%%%%%%%%%%%%%%%%%%%%%%%%%%%%%%%%%%
\section{A relaxed solution concept}\label{sec:relax}
The example in the previous section shows that the limit of the composition $\ell_n\circ\hat{t}_n$ does in general not coincide 
with the composition of the separate limits $\ell\circ\hat{t}$ in jumps. Hence, if we aim for a solution concept 
which is stable w.r.t.\ weak* convergence of the external loads $\ell_n$, we have to weaken the requirements on $\hat{\ell}$.
To be more precise, the above example shows that one can hardly impose any condition on $\hat \ell$ in jumps. 
For this reason, we drop the compatibility condition in \eqref{eq:defell} and replace it by a less restrictive condition
that does not provide any information on $\hat\ell$ in jumps.

\begin{definition}[Relaxed solution concept]\label{def:rexsol}
    Let $\ell \in BV([0,T];\R^d)$ be given. We call triple 
    $(S,\hat{t},\hat{z},\hat{\ell})\in(0,\infty)\times W^{1,\infty}(0,S)\times W^{1,\infty}(0,S;\R^d)\times BV([0,S];\R^d)$ 
	is called \emph{relaxed, normalized, $\mathfrak{p}$-parametrized BV solution} (or short simply \emph{relaxed solution}) of the 
	rate-independent system \eqref{eq:ris} associated with $\ell$, if 
	it satisfies \eqref{eq:initendcond}--\eqref{eq:energy} and, instead of the compatibility condition in \eqref{eq:defell}, we just require 
    \begin{equation}\label{eq:newdefell}
	    \hat{\ell}(s)=\ell(\hat{t}(s))~\mbox{for almost all } s\in M,
    \end{equation} 
    where $M$ is the set, where $\hat{t}$ is increasing, i.e., 
    \begin{equation}\label{eq:defsetM}
        M=\{s\in (0,S) : \hat{t}(s_1)<\hat{t}(s_2)~\mbox{for all } s_1, s_2 \in [0,S] \text{ with } s_1<s<s_2\}.
    \end{equation}
    Herein, we choose the continuous representative of $\hat{t}$ denoted by the same symbol for simplicity.
\end{definition}

The above definition indeed represents a generalization of the original solution concept in Definition~\ref{def:paramsol} as the next result shows.

\begin{lemma}\label{lem:rexparamsol}
    Let $\ell \in BV([0,T];\R^d)$ be given and suppose that 
    $(S,\hat{t},\hat{z},\hat{\ell})\in(0,\infty)\times W^{1,\infty}(0,S)\times W^{1,\infty}(0,S;\R^d)\times BV([0,S];\R^d)$ 
    is a normalized, $\mathfrak{p}$-parametrized BV solution in the sense of Definition~\ref{def:paramsol}. 
    Then it is also a relaxed solution in the sense of Defintion~\ref{def:rexsol}.
\end{lemma}

\begin{proof}
    We only have to verify the condition in \eqref{eq:newdefell}.
    For this purpose, let $s_*\in M$ be arbitrary and set $t_* := \hat t(s_*)$. By definition of $M$, $\hat t^{-1}(t_*)$ is a singleton and 
    therefore \eqref{eq:defell} implies $\hat \ell(s_*) \in \{\ell(t_*), \ell(t_*-), \ell(t_*+)\}$. Since $\ell$ is of bounded variation, it only 
    has countably many jumps collected in the set $J(\ell)\subset [0,T]$ and is continuous elsewhere.
    Since $J(\ell)$ is countable and $\hat t$ is one-to-one on $M$, $\hat t^{-1}(J(\ell))\cap M$ has zero measure and thus
    $\ell(\hat t(s)) = \ell(\hat t(s)-) = \ell(\hat t(s)+)$ f.a.a.\ $s\in M$, which implies \eqref{eq:newdefell}.
\end{proof}

Thanks to Lemma~\ref{lem:rexparamsol}, the existence result from \cite[Prop. 4.2]{KneesZanini} ensuring the existence of a 
normalized, $\mathfrak{p}$-parametrized BV solution immediately yields the existence of a relaxed solution.

\begin{corollary}
    Let $\ell \in BV([0,T];\R^d)$ be arbitrary. Then there exists at least one relaxed solution of \eqref{eq:ris} in the sense of 
    Definition~\ref{def:rexsol}.
\end{corollary}

\subsection{Stability of the relaxed solution concept}\label{sec:stability}
With our relaxed solution concept at hand we are able to prove the following stability result.

\begin{theorem}\label{thm:stability}
    Let $(\ell_n)_{n\in \N}$ be a bounded sequence in $BV([0,T];\R^d)$ and consider a sequence 
    $(S_n,\hat{t}_n,\hat{z}_n,\hat{\ell}_n)$ of relaxed solutions associated with $\ell_n$. 
    Assume moreover that the sequence $(\|\hat \ell_n\|_{BV([0,S_n];\R^d)})_{n\in \N}$ is bounded.
    Then there exists a subsequence (also denoted with the index $n$) such that 
	\begin{equation}\label{eq:weakstarconv}
		S_n\to S,\quad\hat{z}_n\rightharpoonup^* \hat{z}~\mbox{in}~W^{1,\infty}(0,S;\R^d),
		\quad\hat{t}_n\rightharpoonup^* \hat{t}~\mbox{in}~W^{1,\infty}(0,S)	
	\end{equation}
	and
    \begin{alignat}{6}
        \hat{\ell}_n & \weakly^*\hat{\ell} & \quad & \text{in } \BV(0,S;\R^d),  \quad  & \ell_n & \weakly^* \ell & \quad & \text{in } \BV(0,T;\R^d), \\
        \hat{\ell}_n(s) &\to \hat{\ell}(s) &\quad & \text{for all } s\in[0,S], \quad & \ell_n(t) & \to \ell(t) &\quad & \text{for all } t\in[0,T]. \label{eq:ptconv}
    \end{alignat}
	Herein the functions are constantly extended if $S_n<S$.
	Furthermore, the limit $(S,\hat{t},\hat{z},\hat{\ell})$ is a relaxed solution associated with $\ell$.
\end{theorem}

\begin{proof}
	\textit{1. Convergence of a subsequence}\\
	We start with the boundedness of the sequence of relaxed solutions. First of all, \eqref{eq:normalization} gives  
	$\|\hat{z}'\|_{L^\infty(0,S_n;\R^d)}\leq\frac{1}{c}$ with $c>0$ from \eqref{R3}. 
	Next we prove that the artificial end time is uniformly bounded. Thanks to the boundedness of $\hat{z}_n'$ and $\hat{\ell}_n$
	(by assumption), the energy identity \eqref{eq:energy} yields
	\begin{align*}
		& \EE(\hat{z}_n(S_n))+\int_0^{S_n}\RR(\hat{z}_n'(r))+\|\hat{z}_n'(r)\|\abst(-D_z\hat{\II}_n(r,\hat{z}_n(r)),\partial\RR(0)) \, dr\\
		& \quad = \EE(z_0)+\int_0^{S_n}\langle \hat{\ell}_n(r),\hat{z}_n'(r)\rangle \, dr
		\leq \EE(z_0)+\|\hat{z}_n'\|_{L^\infty(0,S_n;\R^d)}\|\hat{\ell}_n\|_{L^1(0,S_n;\R^d)}\leq C.
	\end{align*}
	In combination with \eqref{eq:normalization}, this implies
	\begin{align*}
		S_n &=\int_0^{S_n}\hat{t}_n'(r)+\RR(\hat{z}_n'(r))+\|\hat{z}_n'(r)\|\abst(-D_z\hat{\II}_n(r,\hat{z}_n(r)),\partial\RR(0))\, dr\\
		&= T+ \int_0^{S_n}\RR(\hat{z}_n'(r))+\|\hat{z}_n'(r)\|\abst(-D_z\hat{\II}_n(r,\hat{z}_n(r)),\partial\RR(0))\, dr \leq C,
	\end{align*}
	where we made use of the non-negativity of $\EE(\hat{z}_n(S_n))$ by assumption~\eqref{eq:F1}. 
	Therefore, there is a subsequence (denoted by the same index) such that $S_n \to S$. 
	As in the statement of the theorem, we extend $\hat t_n$, $\hat z_n$, and $\hat\ell_n$ by constant continuation, 
	if necessary, i.e., if $S_n < S$. Due to $\|\hat{z}_n'\|_{L^\infty(0,S;\R^d)}\leq 1/c$ and $\hat{z}_n(0)=z_0$ for all $n \in \N$, 
	we have $\|\hat{z}_n\|_{W^{1,\infty}(0,S;\R^d)}\leq C$. Furthermore, due to \eqref{eq:initendcond}-\eqref{eq:normalization} 
	we conclude $\|\hat{t}_n\|_{W^{1,\infty}(0,S)}\leq C$. Thus, there exists a further subsequence (again denoted by the same symbol) 
	such that \eqref{eq:weakstarconv} holds. The weak$*$ and the pointwise convergence of the sequences of loads 
	along a further subsequence is an immediate consequence of the boundedness of the respective sequences by assumption.\\

	\textit{2. Correlation between $\hat{\ell}$ and $\ell$}\\
	Next we show that \eqref{eq:newdefell} holds true. We argue by contradiction and assume that the set 
    \begin{equation}\label{eq:defG}
        G:=\{s\in M:\hat{\ell}(s)\neq \ell(\hat{t}(s))\}\subset M
    \end{equation}    	
	has positive measure. This yields the existence of a $\rho >0$ such that 
	\begin{equation*}
		\widetilde{M}_\rho = M\cap\{s\in (0,S): \|\hat{\ell}(s)-\ell(\hat{t}(s))\|\geq\rho\}
	\end{equation*}
	has positive measure, too, since if not, then the sets $\{s\in M: \|\hat{\ell}(s)-\ell(\hat{t}(s))\|\geq \frac{1}{n}\}$ have measure zero for all $n\in \N$ 
	and thus $G= \bigcup_{n\in\N}\{s\in M: \|\hat{\ell}(s)-\ell(\hat{t}(s))\|\geq \frac{1}{n}\}$ as a countable union of null sets, too, 
	which contradicts the assumption.
	Since $\hat{t}_n,\hat{\ell}_n$ are part of the relaxed solution associated with $\ell_n$, we have $\hat{\ell}_n(s)=\ell_n(\hat{t}_n(s))$ 
	a.e.\ in $M_n :=\{s\in (0,S) : \hat{t}_n(s_1)<\hat{t}_n(s_2)~\mbox{for all }s_1<s<s_2\}$. We denote the corresponding null sets, 
	where this is not satisfied, by $N_n$. Then 
    \begin{equation*}
        M_\rho := \widetilde{M}_\rho \setminus\bigcup_{n\in\N}N_n
    \end{equation*}    	
    still has positive measure.
	
	Now let $k\in\N$, $k > 1$, be arbitrary. Since $M_\rho$ has positive measure, there exists points 
	$\tilde s_1<\tilde s_2<\ldots<\tilde s_{2k}\in M_\rho$ and, by construction of $M$, there holds 
    \begin{equation*}
        \hat t(\tilde s_{2 i}) < \hat t(\tilde s_{2 (i+1)})\quad \forall\, i = 1, ..., k-1.
    \end{equation*}    	
	Therefore, if we define $s_i := \tilde s_{2i}$, $i=1, ..., k$, then $s_1 < s_2 < \ldots < s_k \in M_\rho$ holds true as well as 
    $\hat{t}(s_1)<\hat{t}(s_2)<\ldots<\hat{t}(s_k)$.  Thus we obtain
	\begin{equation}
		\delta_k :=\min_{i=1,\dots,k-1}\frac{\hat{t}(s_{i+1})-\hat{t}(s_i)}{4}>0 \quad \text{and} \quad  
		\mu_k :=\min_{i=1,\dots,k-1}\frac{s_{i+1}-s_i}{4}>0. \label{eq:defdeltamu}
	\end{equation}
	
	Next we verify that for all $i=1, ..., k$
	\begin{equation}
	    \forall \, \epsilon>0~ \exists  \, \bar{n}_i(\epsilon) \in\N~ \forall \, n\geq \bar{n}_i(\epsilon):~ \lambda(M_n\cap B_\epsilon(s_i))>0, \label{eq:claim}
	\end{equation}
	which allows us to choose $s_i^n\in B_\epsilon(s_i)$ with $\hat{\ell}_n(s_i^n)=\ell_n(\hat{t}_n(s_i^n))$.
	Assuming that \eqref{eq:claim} is not valid implies the existence of an $\epsilon>0$ and 
	a subsequence $(n_j)_{j\in \N}$ such that $\lambda(M_{n_j}\cap B_\epsilon(s_i))=0$ for all $j\in\N$. Then, by the definition of 
	$M_{n_j}$, we obtain for all $s\in B_\epsilon(s_i)\setminus M_{n_j}$ an interval $I_s\ni s$ with $\hat{t}_{n_j}= \text{const.}$ on $I_s$. 
	As $\lambda(M_{n_j}\cap B_\epsilon(s_i))=0$, such an interval exists for almost all $s \in B_\epsilon(s_i)$ and thus, 
	as $\hat{t}_{n_j}$ is absolutely continuous, $\hat{t}_{n_j}$ is constant on $B_\epsilon(s_i)$.
	Along with the uniform convergence $\hat{t}_{n_j}\to\hat{t}$, this implies
	\begin{equation*}
		\hat{t}\Big(s_i-\frac{\epsilon}{2}\Big)
		= \lim_{j\to\infty} \hat{t}_{n_j}\Big(s_i-\frac{\epsilon}{2}\Big)
		=\lim_{j\to\infty} \hat{t}_{n_j}\Big(s_i+\frac{\epsilon}{2}\Big)
		=\hat{t}\Big(s_i+\frac{\epsilon}{2}\Big), 
	\end{equation*}  
    which contradicts $s_i\in M_\rho \subset M$ so that \eqref{eq:claim} is indeed true.
    
 	Therefore, if we now choose $\epsilon = \epsilon_k := \min\{\delta_k,\mu_k\}$ in \eqref{eq:claim}, 
	we obtain for all $n\geq \bar{n}_k =\max\{\bar{n}_i(\epsilon_k):i=1,\dots,k\}$, with $\bar{n}_i(\epsilon_k)$ from \eqref{eq:claim}, 
	and all $i=1,\dots,k$ a point $s_i^n\in (B_{\epsilon_k}(s_i)\cap M_n)\setminus N_n$, which implies 
    \begin{equation}\label{eq:sinMn}
        \hat{\ell}_n(s_i^n)=\ell_n(\hat{t}_n(s_i^n)) \quad \forall\, n \geq \bar n_k
    \end{equation}    	
	by the definition of $M_n$ and $N_n$. Note that $N_n$ has measure zero.
    Together with \eqref{eq:sinMn}, the triangle inequality gives
    \begin{equation}\label{eq:triangle}
        \sum_{i=1}^k \|\hat{\ell}_n(s_i)-\ell_n(\hat{t}(s_i))\|
        \leq \sum_{i=1}^k \|\ell_n(\hat{t}_n(s_i^n))-\ell_n(\hat{t}(s_i))\|+ \sum_{i=1}^k \|\hat{\ell}_n(s_i^n)-\hat{\ell}_n(s_i)\| 
    \end{equation}
    for all $n \geq \bar n_k$. 	
	
	Let us estimate the expressions on the right hand side of \eqref{eq:triangle}.
    We already know that the limit $\hat t$ satisfies $\hat t \in W^{1,\infty}(0,S)$ and hence it is Lipschitz continuous with constant $L>0$ on $[0,S]$. 
    Below we will see that $\hat t$ and $\hat z$ satisfy the normalization condition in \eqref{eq:normalization} 
    and thus, the Lipschitz constant of $\hat t$ is less or equal one.
	Moreover, due to the compact embedding $W^{1,\infty}(0,S)\hookrightarrow C([0,S])$, 
	$\hat{t}_n$ converges uniformly to $\hat{t}$ in $[0,S]$ and hence there exists an index $\tilde{n}_k \in\N$ such that
	$\sup_{s\in[0,S]}|\hat{t}_n(s)-\hat{t}(s)|<\delta_k$ for all $n\geq\tilde{n}_k$.
    Together with the Lipschitz continuity of $\hat t$ with constant one and $s_i^n\in B_{\epsilon_k}(s_i) \subset B_{\delta_k}(s_i)$ 
    for $n \geq \bar n_k$, this results in
	\begin{equation}
		|\hat{t}_n(s_i^n)-\hat{t}(s_i)|\leq |\hat{t}_n(s_i^n)-\hat{t}(s_i^n)|+|\hat{t}(s_i^n)-\hat{t}(s_i)|<2\delta_k. \label{eq:estimatettn}
	\end{equation}
	for all $n\geq n_k := \max\{\bar{n}_k,\tilde{n}_k\}$.
    In view of the definition of $\delta_k$ in \eqref{eq:defdeltamu}, this estimate implies that the intervals 
	$[\min\{\hat{t}_n(s_i^n),\hat{t}(s_i)\}, \max\{\hat{t}_n(s_i^n),\hat{t}(s_i)\}]$, $i=1,\dots,k$, do not overlap and 
	consequently 
	\begin{equation}\label{eq:varest1}
	    \sum_{i=1}^k \|\ell_n(\hat{t}_n(s_i^n))-\ell_n(\hat{t}(s_i))\|
	    \leq \mbox{Var}(\ell_n;[0,T]) \quad \forall \, n \geq n_k.
	\end{equation}
	Similarly, because of $s_i^n\in B_{\epsilon_k}(s_i) \subset B_{\mu_k}(s_i)$ and the definition of $\mu_k$ in \eqref{eq:defdeltamu}, 
	the intervals $[\min\{s_i^n,s_i\}, \max\{s_i^n,s_i\}]$, $i=1, ..., k$, do not overlap and therefore 
	\begin{equation}\label{eq:varest2}
    	\sum_{i=1}^k \|\hat{\ell}_n(s_i^n)-\hat{\ell}_n(s_i)\| \leq \mbox{Var}(\hat{\ell}_n;[0,S]) \quad \forall \, n \geq n_k.
	\end{equation}
	In view of the 	boundedness of $(\mbox{Var}(\ell_n; [0,T]))_n$ and $(\mbox{Var}(\hat\ell_n; [0,S]))_n$ by assumption,
	inserting \eqref{eq:varest1} and \eqref{eq:varest2} in \eqref{eq:triangle} yields 
    \begin{equation*}
        \sum_{i=1}^k \|\hat{\ell}_n(s_i)-\ell_n(\hat{t}(s_i))\| 
        \leq \mbox{Var}(\ell_n;[0,T])+\mbox{Var}(\hat{\ell}_n;[0,S]) \leq C \quad \forall\, n \geq n_k.
    \end{equation*}    	
    with $C>0$ independent of $n$ and $k$.
	Now using the pointwise convergences of $\hat{\ell}_n,\ell_n$ by \eqref{eq:ptconv} and passing to the limit $n\to \infty$ results in 
	\begin{equation*}
		\sum_{i=1}^k \|\hat{\ell}(s_i)-\ell(\hat{t}(s_i))\|\leq C.
	\end{equation*}
	On account of $s_i\in M_\rho$, however,  the left hand side is larger than $k\,\rho$ and, as  $k$ was arbitrary, 
	passing to the limit $k\to \infty$ finally leads to the desired contradiction.
	Hence, $G$ as defined in \eqref{eq:defG} has indeed measure zero and consequently \eqref{eq:newdefell} is valid.
	
	Proving the remaining conditions \eqref{eq:initendcond}--\eqref{eq:energy} for the limit $(S, \hat t, \hat z, \hat\ell)$ 
	is along the lines of the proof of stability for normalized, $\mathfrak{p}$-parametrized BV solutions for problems with smooth external loads, 
	see e.g.\ \cite[Sec.~5]{MRS12}. However, for convenience of the reader, we present the arguments in detail.\\ 
	
	\textit{3.  Initial and end time condtion, monotony of $\hat{t}$}\\
	Since the set $\{f\in L^{\infty}(0,S):f(s)\geq 0~ f.a.a.~s\in(0,S)\}$ is weakly* closed, 
	the limit fulfills $\hat{t}'(s)\geq0$ a.e.\ in $(0,S)$. 
	Moreover, as already used before, $W^{1,\infty}(0,S;\R^d)$ embeds compactly into $C(0,S;\R^d)$ and therefore, 
	$(\hat{t}_n)_{n\in\N}$ and $(\hat{z}_n)_{n\in\N}$ converge uniformly. Therefor, $\hat t(0) = 0$ as well as 
	$\hat{t}(S)=\lim_{n\to\infty}\hat{t}_n(S_n)=T$ and $\hat{z}(0)=\lim_{n\in\N}\hat{z}_n(0)=z_0$ hold true so that 
	\eqref{eq:initendcond} is fulfilled.\\
		
	\textit{4. Energy equality}\\
	Let $s_1, s_2 \in [0,S]$ be given.
	Due to the pointwise convergence of $\hat{z}_n\to\hat{z}$ and the continuity of $\EE$, 
	we have $\EE(\hat{z}_n(s_i))\to \EE(\hat{z}(s_i))$, $i=1,2$.
    Along with the the lower semi-continuity of $\mathfrak{p}(\cdot,\cdot)$ according to \cite[Lemma 3.1]{MRS09}, it follows that
	\begin{equation}
	\begin{aligned}\label{eq:proofenergy}
		&\EE(\hat{z}(s_2))+\int_{s_1}^{s_2}\RR(\hat{z}'(r))+\|\hat{z}'(s)\|\abst(-D_z\hat{\II}(r,\hat{z}(r)),\partial\RR(0))\, dr\\
		& \quad =\EE(\hat{z}(s_2))+\int_{s_1}^{s_2} \mathfrak{p}(\hat{z}'(r),-D_z\hat{\II}(r,\hat{z}(r))) \,dr \\
		&\quad \leq \liminf_{n\to\infty} \bigg(\EE(\hat{z}_n(s_2))+\int_{s_1}^{s_2} \mathfrak{p}(\hat{z}_n'(r),-D_z\hat{\II}_n(r,\hat{z}_n(r)))\, dr\bigg)\\
		&\quad \leq \limsup_{n\to\infty} \bigg(\EE(\hat{z}_n(s_2))+\int_{s_1}^{s_2} \mathfrak{p}(\hat{z}_n'(r),-D_z\hat{\II}_n(r,\hat{z}_n(r)))\, dr\bigg)\\
		& \quad \leq \limsup_{n\to\infty} \bigg(\EE(\hat{z}_n(s_1))+\int_{s_1}^{s_2}\langle \hat{\ell}_n(r),\hat{z}_n'(r)\rangle \, dr\bigg)\\
		&\quad = \EE(\hat{z}(s_1))+\int_{s_1}^{s_2}\langle \hat{\ell}(r),\hat{z}'(r)\rangle\,  dr,
	\end{aligned} 
	\end{equation}
	where we exploited the strong convergence $\hat{\ell}_n\to\ell$ in $L^2(0,S;\R^d)$ by the compactness of 
	$\BV(0,S;\R^d) \embed L^2(0,S;\R^d)$ 
	and the weak* convergence $\hat{z}_n'\weakly^* \hat{z}'$ in $L^\infty(0,S;\R^d)$ in the last step.
	Because the opposite inequality is always fulfilled by Lemma~\ref{lem:ygrene}, we end up with equality and since $s_1, s_2\in[0,S]$ 
	were arbitrary, \eqref{eq:energy} is shown. \\
	
	\textit{5. Complementary condition}\\
	The convergences $\hat{\ell}_n\to\hat{\ell}$ in $L^2(0,S;\R^d)$ and $\hat{z}_n\rightharpoonup^*\hat{z}$ in 
	$W^{1,\infty}(0,S;\R^d)$ lead to $D_z\hat{\II}_n(\cdot,\hat{z}_n(\cdot))\to D_z\hat{\II}(\cdot,\hat{z}(\cdot))$ in $L^2(0,S;\R^d)$ and, 
	by Lipschitz-continuity of the distance, we obtain
	\begin{align*}
		\abst\big(-D_z\hat{\II}_n(\cdot,\hat{z}_n(\cdot)),\partial\RR(0)\big)\to \abst\big(-D_z\hat{\II}(\cdot,\hat{z}(\cdot)),\partial\RR(0)\big) \mbox{ in } L^2(0,S).
	\end{align*}
    Together with the weak* convergence $\hat{t}_n'\rightharpoonup^*\hat{t}'$ in $L^\infty(0,S)$, this results in
	\begin{align*}
		 & \int_0^S \hat{t}'(r)\abst(-D_z\hat{\II}(r,\hat{z}(r)),\partial\RR(0)) dr \\
		 & \quad =\lim_{n\to\infty} \int_0^S \hat{t}_n'(r)\abst(-D_z\hat{\II}_n(r,\hat{z}_n(r)),\partial\RR(0)) dr =0
	\end{align*}
	such that $\hat{t}'(s)\abst(-D_z\hat{\II}(s,\hat{z}(s)),\partial\RR(0))=0$ for almost all $s\in(0,S)$ 
	due to the non-negativity of the integrand.\\
	
	\textit{6. Normalization}\\
	The pointwise convergence $\EE(\hat{z}_n(\cdot))\to\EE(\hat{z}(\cdot))$ in combination with the fact that all inequalities 
	in \eqref{eq:proofenergy} hold true with equality yields for all $s_1, s_2\in[0,S]$ that 
	\begin{multline*}
		s_2 - s_1 = \int_{s_1}^{s_2} \hat{t}_n'(r)+ \mathfrak{p}(\hat{z}_n'(r),-D_z\hat{\II}_n(r,\hat{z}_n(r)))\, dr\\
	 	\to \int_{s_1}^{s_2} \hat{t}'(r)+ \mathfrak{p}(\hat{z}'(r),-D_z\hat{\II}(r,\hat{z}(r))) \, dr \quad \text{as } n \to \infty. 
	\end{multline*}
	Thus, assuming the existence of a Lebesgue measurable set $E\subset (0,S)$ 
	with $\hat{t}'(s)+ \mathfrak{p}(\hat{z}'(s),-D_z\hat{\II}(s,\hat{z}(s)))>1$  for almost all $s\in E$ and 
	$\lambda(E)>0$ gives for every finite union $U\subset[0,S]$ of pairwise disjoint intervals with $E\subset U$
	\begin{align*}
		\lambda(U)=\int_U \hat{t}'(r)+ \mathfrak{p}(\hat{z}'(r),-D_z\hat{\II}(r,\hat{z}(r)))\, dr
		& \geq \int_E \hat{t}'(r)+ \mathfrak{p}(\hat{z}'(r),-D_z\hat{\II}(r,\hat{z}(r)))\, dr \\
		& \geq \lambda(E)+\varepsilon
	\end{align*} 
	for an $\varepsilon>0$, only depending on $E$ and not on $U$, which contradicts the regularity of the Lebesgue measure. 
	Similarly $\hat{t}'(s)+ \mathfrak{p}(\hat{z}'(s),-D_z\hat{\II}(s,\hat{z}(s)))<1$ for almost all $s\in E$ with $\lambda(E)>0$ would imply 
	\begin{align*}
		S&=\int_0^S \hat{t}'(r)+ \mathfrak{p}(\hat{z}'(r),-D_z\hat{\II}(r,\hat{z}(r)))\, dr \\
		 &=\int_E\hat{t}'(r)+ \mathfrak{p}(\hat{z}'(r),-D_z\hat{\II}(r,\hat{z}(r)))\, dr
		 +\int_{(0,S)\setminus E} \hat{t}'(r)+ \mathfrak{p}(\hat{z}'(r),-D_z\hat{\II}(r,\hat{z}(r)))\,dr\\
		 &<\lambda(E)+\lambda((0,S)\setminus E) = S.
	\end{align*}
	Hence, we have $\hat{t}'(\cdot)+ \mathfrak{p}(\hat{z}'(\cdot),-D_z\hat{\II}(\cdot,\hat{z}(\cdot)))=1$ a.e.\ $(0,S)$, which is \eqref{eq:normalization}.
\end{proof}

A stability result of the form of Theorem~\ref{thm:stability} is of course useful in many regards. 
As a potential application, we consider the following optimal control problem
	\begin{equation}
	\left\{\quad 
		\begin{aligned}\label{eq:optcontrol}
			\min \quad &  J(S,\hat{z},\ell):=j(\hat{z}(S))+\beta \, \|\ell\|_{BV([0,T];\R^d)}\\
		    \text{s.t.} \quad & \ell\in BV([0,T];\R^d),\quad(S,\hat{t},\hat{z},\hat{\ell})\in\LL(\ell),\quad \|\hat{\ell}\|_{BV([0,S];\R^d)}\leq K,
		\end{aligned}
		\right.
	\end{equation}
where 
\begin{align*}
	\LL(\ell) := 
	\big\{ (S,\hat{t},\hat{z},\hat{\ell})\in [T,\infty) \times W^{1,\infty}(0,S) \times W^{1,\infty}(0,S;\R^d)\times BV([0,S];\R^d): \quad &\\
	(S,\hat{t},\hat{z},\hat{\ell}) \text{ is a relaxed solution associated with $\ell$} & \big\}.
\end{align*}
Furthermore, $\beta>0$ and $K > T\|\ell_0\|$ a given constants. 
The function $j:\R^d\to\R$ in the objective is supposed to be continuous and bounded from below.
By means of Theorem~\ref{thm:stability} we can now prove the following result.

\begin{corollary}
	There exists at least one  globally optimal solution
	\begin{multline*}
	 	(S^*,\hat{t}^*,\hat{z}^*,\hat{\ell}^*,\ell^*)\\
	 	\in [T,\infty)\times W^{1,\infty}(0,S^*) \times W^{1,\infty}(0,S^*;\R^d)\times BV([0,S^*];\R^d)\times BV([0,T];\R^d)
		\end{multline*}
		of \eqref{eq:optcontrol}.
\end{corollary}

\begin{proof}
    Based on Theorem~\ref{thm:stability}, the proof follows standard arguments.
	Due to the assumption $-Az_0-\FF(z_0)+\ell_0\in\partial\RR(0)$, the constant functions $\hat{z}\equiv z_0$ and $\hat{\ell}\equiv \ell_0$ 
	together with $\hat{t}=\text{id}$ and $S=T$ generate a normalized, $\mathfrak{p}$-parametrized BV solution associated with $\ell \equiv \ell_0$
	in the sense of Definition \ref{def:paramsol}. Thus it is also a relaxed solution associated with $\ell\equiv \ell_0$.
	Moreover, $\hat{\ell}$ trivially satisfies the boundedness condition in the constraints. 
	Therefore, the feasible set of \eqref{eq:optcontrol} is non-empty and we obtain an infimal sequence 
	$(S_n,\hat{t}_n,\hat{z}_n,\hat{\ell}_n,\ell_n)$ such that $J(S_n,\hat{z}_n,\ell_n)\to J^*$, where 
    \begin{equation*}
        J^* := \inf\Big\{J(S,\hat{z},\ell): \ell\in BV([0,T];\R^d),~(S,\hat{t},\hat{z},\hat{\ell})\in \LL(\ell),~\|\hat{\ell}\|_{BV([0,S];\R^d)}\leq K\Big\}.
    \end{equation*}    	
	Since $j$ is bounded from below, the sequence $(\ell_n)_n$ is bounded in $BV([0,T];\R^d)$. 
	Moreover, according the constraints, the sequence $(\|\hat{\ell}_n\|_{BV([0,S_n];\R^d)})_n$ is also bounded.  
    Therefore Theorem~\ref{thm:stability} gives the existence of a subsequence converging to a limit $(S^*,\hat{t}^*,\hat{z}^*,\hat{\ell}^*,\ell^*)$ 
    such that $(S^*,\hat{t}^*,\hat{z}^*,\hat{\ell}^*)\in\LL(\ell^*)$ in the sense of \eqref{eq:weakstarconv}--\eqref{eq:ptconv}. 
    Now, given an arbitrary partition $\{t_k\}$ of $[0,T]$, the pointwise convergence of $\ell_n$ from \eqref{eq:ptconv} implies
    \begin{equation*}
        \sum_{k} \|\ell(t_k) - \ell(t_{k-1})\| 
        = \lim_{n\to\infty}\sum_k  \|\ell_n(t_k) - \ell_n(t_{k-1})\| \leq \liminf_{n\to\infty} \mbox{Var}(\ell_n;[0,T]).
    \end{equation*}
    Since $\{t_k\}$ was arbitrary, the total variation is thus lower semicontinuous w.r.t.\ pointwise convergence 
    and so is the norm in $BV([0,T];\R^d)$ and $BV([0,S];\R^d)$, respectively.
    Thus the bound on the $BV([0,S];\R^d)$-norm readily carries over to the limit $\hat\ell^*$ giving its feasibility. 
    Finally, the compact 	embedding $W^{1,\infty}(0,S^*)\hookrightarrow C([0,S^*])$ yields $\hat{z}_n\to \hat{z}^*$ 
    uniformly so that exploiting the continuity of $j$ and the lower 	semi-continuity of the $BV([0,T];\R^d)$-norm results in 
    $J(S^*,\hat{z}^*,\ell^*)\leq \liminf_{n\to\infty} J(S_n,\hat{z}_n,\ell_n)=J^*$, which means $(S^*,\hat{t}^*,\hat{z}^*,\hat{\ell}^*,\ell^*)$ is an 
	optimal solution of \eqref{eq:optcontrol}.
\end{proof}

%%%%%%%%%%%%%%%%%%%%%%%%%%%%%%%%%%%%%%%%%%%%%%%%%%%%%%%%%%%%%%%%
%%%%%%%%%%%%%%%%%%%%%%%%%%%%%%%%%%%%%%%%%%%%%%%%%%%%%%%%%%%%%%%%
%%%%%%%%%%%%%%%%%%%%%%%%%%%%%%%%%%%%%%%%%%%%%%%%%%%%%%%%%%%%%%%%
%%%%%%%%%%%%%%%%%%%%%%%%%%%%%%%%%%%%%%%%%%%%%%%%%%%%%%%%%%%%%%%%

\section{Physical plausibility of the relaxed solution concept}\label{sec:plausibility}

\subsection{Relation to local solutions}
In order to classify our relaxed solution concept, we compare it with the concept of a \emph{local solution}, which is 
one of the weakest solution concepts imposing less conditions on a solution compared to the other concepts.
In our finite dimensional setting with loads in $BV$, this solution concept reads as follows:

\begin{definition}\label{def:locsol}
	Let $\ell\in BV([0,T];\R^d)$ be given. We call $z\in BV([0,T];\R^d)$ a \emph{local solution} of \eqref{eq:ris} associated with $\ell$, if
	\begin{align}
		& 0\in\partial\RR(0)+D_z\II(t,z(t)) \quad \text{f.a.a.\ }t\in[0,T]\label{eq:locstab}\\
		& \begin{aligned}[t]
		    & \II(t_2,z(t_2))+\mbox{Diss}_\RR(z;[t_1,t_2]) \\[-0.5ex]
		    & \qquad\qquad \leq \II(t_1,z(t_1))- \int_{t_1}^{t_2} z(r) \, d\ell(r) \quad\forall~0\leq t_1\leq t_2\leq T, 		
		\end{aligned} \label{eq:energyineqlocsol}
	\end{align}
	where 
    \begin{equation}\label{eq:defDissR}
    \begin{aligned}
        & \mbox{Diss}_\RR(z;[t_1,t_2]) := \\
        & \quad \sup\Big\{\sum_{i=1}^{k}\RR(z(\xi_i)-z(\xi_{i-1}))~\big|~t_1=\xi_0<\xi_1<\dots<\xi_n=t_2,~n\in\N\Big\}
    \end{aligned}
    \end{equation}    	
	 and the integral on the left hand side is to be understood as a Kurzweil integral.
\end{definition}

\begin{remark}
    For external loads $\ell\in H^1(0,T;\R^d)$, the Kurzweil integral in \eqref{eq:energyineqlocsol} 
    can be converted into the Lebesgue integral $\int_{t_1}^{t_2} \langle z(r),\ell'(r)\rangle dr$, 
	cf. \cite[Prop. 1.10]{Krejci2009}, so that the above definition coincides with \cite[Def. 3.3.2]{MRRIS}.
\end{remark}	 

We underline that the subadditivity of $\RR$ implies that
\begin{equation}\label{eq:dissRadd}
   \mbox{Diss}_\RR(z;[t_1,t_2]) = \mbox{Diss}_\RR(z;[0,t_2]) - \mbox{Diss}_\RR(z;[0,t_1]),
\end{equation}
for all $0 \leq t_1 \leq t_2 \leq T$ provided that $z$ is continuous in $t_1$.
In order to compare our solution concept with the concept of local solutions, 
we have to translate a relaxed solution into physical time. For this purpose, assume that we are given a relaxed solution 
$(S,\hat{t},\hat{z},\hat{\ell})$. Then we define the set of projections of $(\hat{t},\hat{z})$ as
\begin{equation}
\begin{aligned}
	& \mathfrak{P}(\hat{t},\hat{z}) :=\\
	& \quad \{z:[0,T]\to\R^d : \forall\, t\in[0,T]~\exists\, s\in[0,S] \mbox{ with } (t,z(t))=(\hat{t}(s),\hat{z}(s))\}.\label{eq:defP}
\end{aligned}
\end{equation}
Note that $\mathfrak{P}$ consists of all functions $z:[0,T]\to \R^d$, whose graph is subset of the solution trajectory
$[0,S] \ni s \mapsto  (\hat{t}(s),\hat{z}(s))$.
The following theorem correlates local with relaxed solutions.

\begin{theorem}\label{thm:locstab}
	Let $\ell\in BV([0,T];\R^d)$ and local solution $z\in BV([0,T];\R^d)$ be given. Furthermore, we assume that $\RR$ is symmetric, i.e., 
    \begin{equation}\label{eq:symm}
        \RR(z_1)=\RR(z_2)\quad \Longleftrightarrow \quad \|z_1\|=\|z_2\| \quad \forall\, z_1,z_2\in\R^d,
    \end{equation}    	
     and $z$ (and its representative fulfilling \eqref{eq:locstab}--\eqref{eq:energyineqlocsol}, respectively) only admits a finite number of jumps.
	Then there exists a parametrization and $\hat{\ell}\in BV([0,S];\R^d)$ such that 
	$(S,\hat{t},\hat{z},\hat{\ell})$ is a relaxed solution and $z\in\mathfrak{P}(\hat{t},\hat{z})$.
\end{theorem}

\begin{proof}
\textit{1. Construction and regularity of $\hat{t}$ and $\hat{z}$}\\
Throughout the proof, we always consider the representative of $z$ satisfying \eqref{eq:locstab}--\eqref{eq:energyineqlocsol}
and denote it by the same symbol for convenience.
We first need to 
construct a suitable parametrization $s\mapsto (\hat t(s), \hat z(s))$ of the solution trajectory as a candidate for a relaxed solution.
In the literature on $\mathfrak{p}$-parametrized BV solutions, given a solution in physical time, 
the parametrized solution $\hat{z}$ is frequently designed in jumps in such a way that the dissipative distance is minimized. 
More precisely, if $t \in [0,T]$ is a jump point between the values $z_1$ and $z_2$, 
then $\hat{z}$ is constructed by means of the minimizer of 
\begin{equation}\label{eq:defdissdist}
    \left\{\quad
    \begin{aligned}
    	\min \quad &\int_0^1\mathfrak{p}(v'(r),-D_z\II(t,v(r)))dr\\
        \text{s.t.} \quad & v\in W^{1,1}(0,1;\R^d), \quad v(0)=z_1,\quad v(1)=z_2,
    \end{aligned}
    \right.
\end{equation}
cf.\ \cite[Sec. 3.8.2]{MRRIS}. 
Here, we pursue a different strategy. Since, in the relaxed solution concept, $\hat\ell$ is an additional variable in jumps, 
we have certain freedom in the choice of $\hat{z}$. This allows us to use a simpler construction of $\hat z$ and the associated parametrization.
For this purpose, let us define the following modified dissipation 
\begin{equation}\label{eq:Diss0}
    \Diss_{0}(z;[0, t])
    := \RR(z(0+) - z_0) + \Diss_{\RR}(z;[0+,t])
\end{equation}
taking into account potential jumps at the initial time. Note that the local solution need not satisfy 
the initial condition. Given $\Diss_0$, we set  
\begin{equation}\label{eq:defs}
 	s:[0,T]\to[0,S], \quad s(t):=
    \begin{cases}
        0, & t = 0,\\
     	t+\Diss_0(z;[0,t]), & t \in (0,T] ,\\
    \end{cases}     	
\end{equation} 
where $S:=s(T)$. Note that, by \eqref{R3}, 
\begin{equation*}
    \Diss_\RR(z;[0,T]) \leq C \, \mbox{Var}(z;[0,T]) < \infty,
\end{equation*}
so that $S$ is finite. Let us investigate the regularity of $s$. We denote the jump points of $z$ that are greater zero by 
$0 < t_1 < t_2 < \ldots < t_N \leq T$ and set $J(z) := \{t_1,\dots,t_N\}$. 
According to its definition, $\Diss_\RR(z;[0+,\cdot\,])$ is continuous in intervals of continuity of $z$, i.e., 
$(t_n,t_{n+1})$ with $t_n,t_{n+1}\in J(z)$, $n = 1, \ldots, N$, as it inherits its continuity from $z$ there. 
Thus, by construction, the jump points of $s$ are $J(s) := \{0, t_1,\dots,t_N\}$, if $z(0+) \neq z_0$, and $J(s) = J(z)$ otherwise. 
In the remaining intervals however, i.e., in $(0,t_1)$, $(t_1, t_2)$, ..., $(t_{N-1}, t_N)$, $(t_N, T)$, the function $s$ is continuous. 
Moreover, the non-negativity of $\RR$ implies that $s$ is strictly increasing and hence invertible. 
We denote its inverse function as 
\begin{equation*}
    \hat{t}:[0,S]\setminus \bigcup_{n=0}^{N}I_n\to [0,T]\setminus J(s)
\end{equation*}
with 
\begin{equation}\label{eq:defIn}
    I_0 := [0, \RR(z(0+) - z_0)] 
    \quad \text{and}  \quad I_n:=[s(t_n-),s(t_n+)].
\end{equation}
Then $\hat{t}$ is monotonously increasing as an inverse function of an increasing function. Furthermore, $\hat t$ is Lipschitz-continuous  
with Lipschitz constant $L\leq 1$ on all intervals of $[0,S]\setminus \bigcup_{n=0}^{N}I_n$, 
since \eqref{eq:defs} and \eqref{eq:dissRadd} give for $s_1<s_2$ from such an interval that
\begin{equation}\label{eq:hattLip}
\begin{aligned}
 	0 \leq \hat{t}(s_2)-\hat{t}(s_1) &= s_2 - s_1 - \big( \Diss_0(z;[0,\hat t(s_2)]) - \Diss_0(z;[0,\hat t(s_1)])\big)\\
 	&= s_2 - s_1 -\mbox{Diss}_{\RR}(z;[\hat{t}(s_1),\hat{t}(s_2)])\leq s_2-s_1.
\end{aligned}
\end{equation}
Therefore, after constant continuation on all $I_n$, $n=0, ..., N$, i.e.,
\begin{equation*}
    \hat t(s) = 0 \quad \forall\, s\in I_0, \qquad \hat t(s)=t_n\quad \forall\, s\in I_n, \; n = 1, \ldots, N,
\end{equation*}
the function $\hat t$ is Lipschitz-continuous  
with Lipschitz constant $L\leq 1$. So there holds $\hat{t}\in W^{1,\infty}(0,S)$, i.e., the required regularity of $\hat t$, and 
$\hat{t}^{\prime}(s)\geq 0$ for almost all $s\in(0,S)$, i.e., the sign condition in \eqref{eq:complementary}.
Moreover, by construction of $\hat{t}$, we obtain that $M$, the set where $\hat{t}$ is strictly increasing, see \eqref{eq:defsetM},
is given by 
\begin{equation}
    M=(0,S)\setminus\bigcup_{n=0}^{N}I_n=:\bigcup_{m=1}^M G_m, \label{eq:defM}
\end{equation}
where $G_m\subset(0,S)$ are open intervals.
  
Given $\hat t$, we define $\hat{z}$ as composition of $z$ and $\hat{t}$ in the parts, where $z$ is continuous and in jumps we choose $\hat{z}$ 
as the linear function that connects the left-hand and right-hand limits of $z$ such that $\hat z$ reads 
 \begin{equation}
	 \begin{aligned}\label{eq:defhatz}
 		\hat{z}(s) :=\begin{cases}
	           		z(\hat{t}(s)), & s\in(0,S]\setminus \bigcup_{n=0}^{N}I_n,\\[0.5ex]
			        z(t_n-)+\frac{z(t_n)-z(t_n-)}{s(t_n)-s(t_n-)}\, (s-s(t_n-)), & s\in I_n^-,\\[1ex]
			         z(t_n)+\frac{z(t_n+)-z(t_n)}{s(t_n+)-s(t_n)}\,(s-s(t_n)), & s\in I_n^+, \\[1ex]
			         z_0+\frac{z(0+)-z_0}{\RR(z(0+) - z_0)}\,s, & s\in I_0, 
			\end{cases}
 	\end{aligned}
 \end{equation}
where $I_n^- :=[s(t_n-),s(t_n)]$, $I_n^+ :=[s(t_n),s(t_n+)]$, and $I_0$ as defined in \eqref{eq:defIn}. 
Note that, by construction, $\hat z(0) = z_0$, i.e., $\hat z$ satisfies the initial condition in \eqref{eq:initendcond}.
Next we show that $\hat z$ constructed in this way is Lipschitz continuous, too. 
Let us first consider an arbitrary interval $G_m$ and let $s_1,s_2\in G_m$ with $s_1<s_2$ be arbitrary. 
Then, by definition of $G_m$, $z$ is continuous on $[\hat t(s_1), \hat t(s_2)]$ and thus, similarly to \eqref{eq:hattLip}, 
the construction of $s$ in \eqref{eq:defs} yields 
 \begin{align}\label{eq:diffs}
 	s(\hat t(s_2))-\hat t(s_2)-s(\hat t(s_1))+\hat t(s_1)%=\mbox{Diss}_{\mathfrak{p}}(z;[\hat t(s_1),\hat t(s_2)])
 	=\mbox{Diss}_{\RR}(z;[\hat t(s_1),\hat t(s_2)]).
 \end{align}
With this at hand, assumption~\eqref{R3} and the construction of $\hat z$ in \eqref{eq:defhatz} yield
\begin{equation}\label{eq:hatzlip1}
\begin{aligned}
 	c\,\|\hat{z}(s_2)-\hat{z}(s_1)\|
 	&\leq \RR(\hat{z}(s_2)-\hat{z}(s_1)) \\
 	&\leq \Diss_\RR(\hat{z};[s_1,s_2]) \\
 	&=\Diss_\RR(z;[\hat{t}(s_1),\hat{t}(s_2)])\\
	&=s_2-\hat{t}(s_2)-s_1+\hat{t}(s_1)\leq s_2-s_1,
\end{aligned} 
\end{equation}
where we exploited the monotonicity of $\hat t$ in the last estimate.
Next let us consider an arbitrary interval $I_n^-$ and arbitrary points $s_1, s_2 \in I_n^-$. 
Using the definition of $s$  in \eqref{eq:defs} and $\Diss_0$ in \eqref{eq:Diss0} and again \eqref{eq:dissRadd}, one obtains
\begin{equation}\label{eq:estsbyR}
\begin{aligned}
    s(t_n) - s(t_n-)
    & = \Diss_{\RR}(z;[0,t_n]) - \Diss_{\RR}(z;[0,t_n-]) \\
    %&=\Diss_{\RR}(z;[t_n-,t_n])\\
    &= \RR(z(t_n)-z(t_n-)).
\end{aligned}
\end{equation} 
Together with the construction of $\hat z$ in jumps according to \eqref{eq:defhatz}, this gives
\begin{equation}\label{eq:hatzlip2}
\begin{aligned}
 	\|\hat{z}(s_2)-\hat{z}(s_1)\|
 	& = \frac{\|z(t_n)-z(t_n-)\|}{s(t_n)-s(t_n-)} \, |s_2-s_1| \\
 	& \leq \frac{\|z(t_n)-z(t_n-)\|}{\RR(z(t_n)-z(t_n-))}\, |s_2-s_1| \leq \frac{1}{c}\,|s_2-s_1|,
\end{aligned}
\end{equation}
where we again used \eqref{R3} for the last estimate. Completely analogously, one derives the same estimate 
for arbitrary $s_1, s_2 \in I_n^+$. In $I_0$, we similarly obtain 
\begin{equation*}
    \|\hat{z}(s_2)-\hat{z}(s_1)\| = \frac{\|z(0+)-z_0\|}{\RR(z(0+)-z_0)}\, |s_2-s_1| \leq \frac{1}{c}\,|s_2-s_1| \quad \forall\, s_1, s_2 \in I_0.
\end{equation*}
Since $\hat{z}$ is additionally continuous by construction, see \eqref{eq:defhatz}, 
this estimate together with \eqref{eq:hatzlip1} and \eqref{eq:hatzlip2} implies that $\hat z$ is Lipschitz-continuous on $[0,S]$ 
with Lipschitz constant $1/c$. Consequently, we obtain the desired regularity of a relaxed solution, i.e., $\hat{z}\in W^{1,\infty}(0,S;\R^d)$.
Furthermore, $\hat{t},\hat{z}$ are constructed in such a way that $z\in\mathfrak{P}(\hat{t},\hat{z})$.  \\
 
\textit{2. Complementary, normalization, and energy identity outside of jumps}\\
Now that we have defined $\hat t$ and $\hat z$, we still need to define $\hat\ell$ outside the set $M$ so that 
\eqref{eq:complementary}--\eqref{eq:energy} are fulfilled. On $M,$ however, $\hat\ell$ is fixed by \eqref{eq:newdefell} 
and set to $\hat\ell = \ell \circ \hat t$. To show \eqref{eq:complementary}, let us assume by contrary that there exists a 
set $E\subset M$ of positive measure such that 
\begin{equation}\label{eq:complcontrary}
	\abst(-D_z\hat{\II}(s,\hat{z}(s)),\partial\RR(0)) > 0 \quad \text{a.e.\ in } E.
\end{equation} 
For the image of $E$ under $\hat t$, we obtain $|\hat t(E)| = \int_E |\hat t'(s)| \, ds > 0$,
since $\hat t$ is Lipschitz continuous and monotonically increasing on $M$. 
Then, in view of $\hat z = z \circ \hat t$ and $\hat\ell = \ell \circ \hat t$ and the definition of $\hat \II$ in \eqref{eq:defhatI}, 
\eqref{eq:complcontrary} implies $\abst(-D_z \II(t, z(s)),\partial\RR(0)) > 0$ a.e.\ in $\hat t(E)$, which contradicts \eqref{eq:locstab}. 
Thus we obtain
\begin{equation}
 	\abst(-D_z\hat{\II}(s,\hat{z}(s)),\partial\RR(0))=0 \quad \text{f.a.a.\ } s\in M\label{eq:dist=0}
\end{equation} 
Moreover, due to constant continuation outside of $M$, there holds $\hat{t}'(s)=0$ for almost all $s\in (0,S)\setminus M$ 
and hence, \eqref{eq:complementary} is valid independent of the choice of $\hat{\ell}$ on $(0,S)\setminus M$. 

Next we show that \eqref{eq:normalization} is satisfied almost everywhere on $M$. By exploiting \eqref{eq:dist=0}, this is
 \begin{align}
 	\hat{t}'(s)+\RR(\hat{z}'(s))=1 \quad \mbox{f.a.a. }s\in M.
 \end{align}
Because $M$ is a finite union of intervals, see \eqref{eq:defM}, this is in turn equivalent to
 \begin{align}
 	\int_{s_1}^{s_2}\hat{t}'(r)+\RR(\hat{z}'(r))dr=s_2-s_1 \quad\forall~ s_1,s_2\in G_m,~m=1,\dots, M.\label{eq:intform}
 \end{align} 
Using \eqref{eq:diffs} and Lemma \ref{lem:DissR} yield for arbitrary $s_1,s_2\in G_m$
\begin{align*}
	\int_{s_1}^{s_2}\hat{t}'(r)+\RR(\hat{z}'(r))dr
	=~&\hat{t}(s_2)-\hat{t}(s_1)+\mbox{Diss}_\RR(\hat{z};[s_1,s_2])\\
	=~&\hat{t}(s_2)-\hat{t}(s_1)+\mbox{Diss}_\RR(z;[\hat{t}(s_1),\hat{t}(s_2)])\\
	=~&s(\hat{t}(s_2))-s(\hat{t}(s_1))\\
	=~&s_2-s_1,
\end{align*}
which is \eqref{eq:intform} such that \eqref{eq:normalization} indeed holds a.e.\ in $M$.

Concerning the energy identity in \eqref{eq:energy}, take an arbitrary interval $G_m$ and let $s_1, s_2 \in G_m$, $s_1 < s_2$, be arbitrary, too.
By definition of $\hat z$ in \eqref{eq:defhatz}, there holds $\hat z = z \circ \hat t$ in $G_m$ and therefore, 
the energy inequality \eqref{eq:energyineqlocsol} with $t_1 = \hat t(s_1)$ and $t_2 = \hat t(s_2)$ yields
\begin{equation}\label{eq:energy1}
\begin{aligned}
    & \II(\hat{t}(s_2),\hat{z}(s_2))+\mbox{Diss}_\RR(\hat z;[s_1,s_2]) - \II(\hat{t}(s_1),\hat{z}(s_1))\\
    &\qquad = \II(\hat{t}(s_1),z(\hat t(s_2)))+\mbox{Diss}_\RR(z;[\hat t(s_1),\hat t(s_2)]) - \II(\hat{t}(s_1),z(\hat{t}(s_1)))\\
    &\qquad \leq - \int_{\hat t(s_1)}^{\hat t(s_2)} z(t) \, d \ell(t)\\
    &\qquad = - \int_{s_1}^{s_2} (z \circ \hat{t})(r) \, d(\ell\circ \hat{t})(r) = - \int_{s_1}^{s_2} \hat{z}(r) \, d\hat{\ell}(r),
 \end{aligned}
\end{equation}
where we used the change of variables formula for Kurzweil integrals from \cite[Thm. 6.1]{Bensimhoun}, which is applicable 
due to the continuity of $\hat t$. Note that we again used $\hat\ell = \ell\circ \hat t$ in $G_m\subset M$.
For the dissipation, Lemma~\ref{lem:DissR} along with \eqref{eq:dist=0} yields
\begin{equation}\label{eq:energy2}
    \Diss_{\RR}(\hat z;[s_1, s_2]) = \int_{s_1}^{s_2} \RR(\hat z'(r)) +\|\hat{z}'(r)\| \abst(-D_z\hat{\II}(r,\hat{z}(r)),\partial\RR(0))\, dr.
\end{equation}
Since $\hat z \in W^{1,\infty}(0,S;\R^d)$, we are allowed to apply the formula of integration by parts according to \cite[Prop.~1.12]{Krejci2009} to 
the Kurzweil integral on the right hand side of \eqref{eq:energy1}, which results in 
\begin{equation*}
    - \int_{s_1}^{s_2} \hat{z}(r) \, d\hat{\ell}(r)
    =  \int_{s_1}^{s_2} \langle \hat{\ell}(r),\hat{z}'(r)\rangle \, d r + \dual{\hat\ell(s_2)}{\hat z(s_2)} - \dual{\hat\ell(s_1)}{\hat z(s_1)},
\end{equation*}
where we have already rewritten the Kurzweil integral on the right hand side as a Lebesgue integral according to \cite[Prop.~1.10]{Krejci2009} 
and the regularity of $\hat{z}$.
Inserting this together with \eqref{eq:energy2} in \eqref{eq:energy1} implies in view of the definition of $\EE$ in \eqref{eq:defEE} that
\begin{equation}\label{eq:energy3}
\begin{aligned}
    \EE(\hat{z}(s_2))+\int_{s_1}^{s_2} \RR(\hat{z}'(r))+\|\hat{z}'(r)\|\abst(-D_z\hat{\II}(r,\hat{z}(r)),\partial\RR(0)) \, dr \qquad\qquad & \\[-1ex]
    \leq \EE(\hat z(s_1))+\int_{s_1}^{s_2} \langle \hat{\ell}(r),\hat{z}'(r)\rangle \, dr,
\end{aligned}    
\end{equation}
Lemma~\ref{lem:ygrene} yields that \eqref{eq:energy3} is even satisfied with equality and thus, we obtain the desired energy identity 
for arbitrary $s_1, s_2 \in G_m$.\\

\textit{3. Construction of $\hat{\ell}$ in jumps}\\
To motivate our definition of $\hat\ell$ in jumps, let us take a closer look at the connection of the normalization condition \eqref{eq:normalization} 
and the energy identity in jumps. As $\hat t$ is constant there, \eqref{eq:normalization} reads
\begin{equation}\label{eq:normalinjumps}
	\RR(\hat{z}'(s))+\|\hat{z}'(s)\|\abst(-D_z\hat{\II}(s,\hat{z}(s)),\partial\RR(0))=1 \quad \text{f.a.a.\ } s\in \bigcup_{n=0}^{N} I_n.
\end{equation}
The term on the left hand side is exactly the expression that also arises in the energy identity. 
So, if we consider an arbitrary jump interval $I_n$ and arbitrary $s_1, s_2 \in I_n$ and if we suppose that 
$\hat\ell$ is chosen in $I_n$ such that \eqref{eq:normalinjumps} holds, then the energy identity in \eqref{eq:energy} becomes
\begin{equation*}
\begin{alignedat}{3}
    s_2 - s_1 
    & = \int_{s_1}^{s_2} \RR(\hat{z}'(r))+\|\hat{z}'(r)\|\abst(-D_z\hat{\II}(r,\hat{z}(r)),\partial\RR(0)) \, dr  \\
    & = \EE(\hat{z}(s_1)) - \EE(\hat z(s_2)) + \int_{s_1}^{s_2} \langle \hat{\ell}(r),\hat{z}'(r)\rangle \, dr & & \text{(by \eqref{eq:energy})}\\
    & = -\int_{s_1}^{s_2}\langle D_z\EE(\hat{z}(t)),\hat{z}'(t)\rangle\, dt + \int_{s_1}^{s_2} \langle \hat{\ell}(r),\hat{z}'(r)\rangle \, dr \\
    & = -\int_{s_1}^{s_2}\langle D_z\hat{\II}(r,\hat{z}(r)),\hat{z}'(r)\rangle \, dr,
\end{alignedat}    
\end{equation*}
where we used the definition of $\hat\II$ in \eqref{eq:defhatI}. Since this equality must hold for every $s_1, s_2 \in I_n$, provided that $\hat\ell$ 
is such that the energy identity holds, we see that $\hat\ell$ satisfying the normalization condition \eqref{eq:normalinjumps} fulfills the 
energy identity, if and only if
\begin{equation}\label{eq:conditionell}
	\langle D_z\hat{\II}(s,\hat{z}(s)),\hat{z}'(s)\rangle+1=0 \quad \text{f.a.a.\ } s\in \bigcup_{n=0}^{N} I_n.
\end{equation}
To summarize, in order to ensure \eqref{eq:normalization} and \eqref{eq:energy} to hold in jumps, we must construct $\hat\ell$ such that 
\eqref{eq:normalinjumps} and \eqref{eq:conditionell} are fulfilled. In view of $\hat \II(z) = \frac{1}{2} \dual{Az}{z} + \FF(z) - \dual{\hat\ell}{z}$, 
\eqref{eq:conditionell} leads to the natural choice 
\begin{equation}
    \hat{\ell}(s) := A\hat{z}(s)+D_z\FF(\hat{z}(s))+\frac{\hat{z}'(s)}{\|\hat{z}'(s)\|^2} 
    \quad \forall\,  s\in \bigcup_{n=0}^{N} \mathring{I}_n\,,\label{eq:defhatell}
\end{equation}
as \eqref{eq:conditionell} and thus the energy identity are trivially fulfilled then. Note that 
$\hat{z}'(s)$ is constant and different from zero in $\bigcup_{n=0}^{N} \mathring{I}_n$ such that $\hat{\ell}(s)$ is well defined 
and due to the regularity of $\hat z$ an element of $W^{1,\infty}(I_n;\R^d)$ for all $n = 0, ..., N$.
Note moreover that $\mathring{I}_n = \emptyset$, if $\hat z'$ vanishes there.
It remains to prove \eqref{eq:normalinjumps} for $(\hat{t},\hat{z})$ together with $\hat{\ell}$ as defined in \eqref{eq:defhatell}.
For that purpose, let us investigate the dissipation functional in the jumps.
As seen in \eqref{eq:estsbyR}, in a jump point $t_n$, there holds $\RR(z(t_n)-z(t_n-)) = s(t_n)-s(t_n-)$ and completely analogously, one shows 
$\RR(z(t_n+)-z(t_n)) = s(t_n+)-s(t_n)$. Using the monotonicity of $s$ along with the positive 1-homogeneity of $\RR$ 
and the structure of $\hat{z}$ in the jump parts $I_n^-,I_n^+$, see \eqref{eq:defhatz}, this implies
\begin{equation}
	\RR(\hat{z}'(s))=\RR\bigg(\frac{z(t_n)-z(t_n-)}{s(t_n)-s(t_n-)}\bigg) = 1 \quad \forall \, s\in I_n^-\label{eq:estR1}
\end{equation} 
as well as 
\begin{equation}
	\RR(\hat{z}'(s))=\RR\bigg(\frac{z(t_n+)-z(t_n)}{s(t_n+)-s(t_n)}\bigg) = 1 \quad \forall \, s\in I_n^+,\label{eq:estR2}
\end{equation} 
while, on $I_0$, we have by construction of $\hat z$ in \eqref{eq:defhatz} that $\RR(\hat z'(s)) = 1$ for all $s\in I_0$.
Now let $s\in \bigcup_{n=0}^N \mathring{I}_n$ be arbitrary.
Inserting \eqref{eq:defhatI} in the left hand side in \eqref{eq:normalinjumps} and employing Lemma~\ref{lem:Rstar} yield
\begin{align*}
	& \RR(\hat{z}'(s))+\|\hat{z}'(s)\|\abst(-D_z\hat{\II}(s,\hat{z}(s)),\partial\RR(0))\\
	& \quad =\RR(\hat{z}'(s))+\|\hat{z}'(s)\|\abst\bigg(\frac{\hat{z}'(s)}{\|\hat{z}'(s)\|^2},\partial\RR(0)\bigg)\\
	& \quad =\RR(\hat{z}'(s))+\sup_{v\in\R^d,\|v\|\leq\|\hat{z}'(s)\|}\bigg(\Big\langle\frac{\hat{z}'(s)}{\|\hat{z}'(s)\|^2},v\Big\rangle-\RR(v)\bigg)\\
	& \quad =
	\RR(\hat{z}'(s))+\sup_{0<\alpha<1} \alpha \sup_{v\in\R^d,\|v\|=\|\hat{z}'(s)\|}\bigg(\Big\langle\frac{\hat{z}'(s)}{\|\hat{z}'(s)\|^2},v\Big\rangle-\RR(v)\bigg)\\
	& \quad =\RR(\hat{z}'(s))+\sup_{0<\alpha<1} \alpha \sup_{v\in\R^d,\|v\|=\|\hat{z}'(s)\|}\bigg(\Big\langle\frac{\hat{z}'(s)}{\|\hat{z}'(s)\|^2},v\Big\rangle-\RR(\hat{z}'(s))\bigg)\\
	& \quad =\Big\langle\frac{\hat{z}'(s)}{\|\hat{z}'(s)\|^2},\hat{z}'(s)\Big\rangle=1,
\end{align*}
where we exploited the symmetry of $\RR$ by assumption, the 1-homogeneity of $\RR$, and \eqref{eq:estR1} and \eqref{eq:estR2}. 
Thus we have verified \eqref{eq:normalinjumps}, which, together with \eqref{eq:conditionell}, which holds by construction of $\hat\ell$ according to
\eqref{eq:defhatell}, implies \eqref{eq:complementary} and \eqref{eq:energy} in the jumps, as explained above. 
Therefore, the energy identity \eqref{eq:energy} holds for all $s_1, s_2\in I_n$, $n = 0, \dots, N$, and all $s_1, s_2\in G_m$, $m = 1,\dots, M$, 
and consequently, the continuity of $\hat z$ implies that it holds for all $s_1, s_2 \in [0,S]$.\\

\textit{4. Regularity of $\hat{\ell}$}\\
To summarize, we have seen that with $\hat\ell$ defined by
\begin{equation}
	\hat{\ell}(s)=\begin{cases}
				\ell(\hat{t}(s)), & s\in [0,S]\setminus \bigcup_{n=0}^{N} I_n,\\
				A\hat{z}(s)+D_z\FF(\hat{z}(s))+\frac{\hat{z}'(s)}{\|\hat{z}'(s)\|^2}, &  \bigcup_{n=0}^{N} \mathring{I}_n,
	\end{cases}, \label{eq:defhatell2}
\end{equation}
all conditions of a relaxed solution from Definition~\ref{def:rexsol} are fulfilled. 
It remains to verify the required smoothness of $\hat\ell$.  Due to $\ell \in BV([0,T];\R^d)$ and the monotonicity of $\hat t$, the 
composition $\ell\circ \hat t$ is a function in $BV([0,S];\R^d)$. 
Moreover, since $\hat z'$ is constant and non-zero in $\mathring{I}_n$, $A\hat{z}(s)+D_z\FF(\hat{z}(s))+\frac{\hat{z}'(s)}{\|\hat{z}'(s)\|^2}$ 
inherits the regularity of $\hat z$ and is therefore continuous. 
Thus, as a piecewise composition of finitely many functions of bounded variation, 
$\hat\ell$ is a function of bounded variation, too, i.e., $\hat{\ell}\in BV([0,S];\R^d)$ as required.
\end{proof}

\begin{remark}
	The symmetry of $\RR$ is necessary to fulfill the normalization condition \eqref{eq:normalization} 
	with $\hat{\ell}$ as chosen in \eqref{eq:defhatell}. This can be seen as follows.
	Let $\RR:\R^2\to\R$ be defined by $\RR(z)=\frac{1}{2}|z_1|+|z_2|$ and suppose that there is a jump interval $I$ with constant $\hat z'$ given by
	$\hat{z}'(s) = (1,\frac{1}{2})^\top$ for all $s\in I$. Then we have $\RR(\hat{z}'(s))=1$ and 
	$\frac{\hat{z}'(s)}{\|\hat{z}'(s)\|^2}=(\frac{4}{5},\frac{2}{5})^\top$ for all $s\in I$. For the distance to $\partial \RR(0)$, we thus obtain in 
	light of \eqref{eq:defhatell} that 
    \begin{equation*}
    \begin{aligned}
        \abst(-D_z\hat{\II}(s,\hat{z}(s)),\partial\RR(0))
        = \abst\Big(\frac{\hat{z}(s)}{\|\hat{z}'(s)\|^2}, [-\tfrac{1}{2},\tfrac{1}{2}] \times [-1,1]\Big) = \frac{3}{10} \quad \forall\, s\in I.
    \end{aligned}
    \end{equation*}    	
	Due to $\hat t'(s) = 0$ for all $s\in I$, we end up with
	\begin{equation*}
		\hat{t}'(s)+\RR(\hat{z}'(s))+\|\hat{z}'(s)\|\abst(-D_z\hat{\II}(s,\hat{z}(s)),\partial\RR(0))=1+\sqrt{\frac{5}{4}}\frac{3}{10}>1 \quad \forall\, s\in I
	\end{equation*}
	such that \eqref{eq:normalization} is indeed violated.
\end{remark} 

\begin{remark}
	The symmetry of $\RR$ together with \eqref{R1} and \eqref{R2} automatically implies that $\RR$ is a scaled version of the norm, 
	i.e., there exists an $\alpha \geq 0$ such that $\RR(\cdot)=\alpha \|\cdot\|$. 
	This can be seen as follows. Let $e\in \R^d$ be a unit vector and define $\alpha :=\RR(e)$. Then, for $y\neq 0$, the symmetry condition in 
	\eqref{eq:symm} implies $\RR\big(\frac{y}{\|y\|}\big)=\RR(e)=\alpha$ such that $\RR(y)=\alpha \|y\|$ for all $y\in \R^d$
	due to the positive homogeneity of $\RR$.
\end{remark}

\begin{remark}
	The reason why we assume in Theorem~\ref{thm:locstab}
	that $z$ has only finitely many jumps is to guarantee the required regularity of $\hat\ell$, i.e., $\hat{\ell}\in BV([0,S];\R^d)$. 
    Note that, according to \eqref{eq:defhatell2}, $\hat\ell$ is continuous except for the points in the set $\bigcup_{n=0}^{N}\{s(t_n-),s(t_n),s(t_n+)\}$ 
    and thus, 
    if there are countably many of those points, the total variation of $\hat\ell$ may become infinite. 
    At least in one dimension, i.e., if $d=1$, there is however no alternative to the choice of $\hat\ell$ in \eqref{eq:defhatell2}, as it is easily seen that
    this construction of $\hat\ell$ in the jumps is the only way to guarantee \eqref{eq:conditionell} and equivalently the energy identity, while 
    $\hat\ell$ is determined by $\ell\circ \hat t$ outside the jumps.
    Therefore, in order to allow for countable many jumps of $z$ without destroying the regularity of $\hat\ell$,
    we need to modify the construction of $\hat z$ in the jumps. Let us again consider the one dimensional case. 
    If one aims to avoid a jump of $\hat\ell$ at $s = s(t_n-)$,  one has to choose $\hat z$ in the jump from $s(t_n-)$ to $s(t_n)$ such that 
    $\ell(\hat{t}(s))=A\hat{z}(s+)+D_z\FF(\hat{z}(s+))+\frac{\hat{z}'(s+)}{|\hat{z}'(s+)|^2}$. Thanks to the continuity of $\hat z$, this 
	can be reformulated as $D_z\II(\hat{t}(s),\hat{z}(s)) + \frac{\hat{z}'(s+)}{|\hat{z}'(s+)|^2} = 0$, which results in the following condition for $\hat z$
	at the beginning of the jump
	\begin{equation}
		\hat{z}'(s+)= \frac{-D_z\II(\hat{t}(s),\hat{z}(s))}{|D_z\II(\hat{t}(s),\hat{z}(s))|^2},\label{eq:conditionz'}
	\end{equation}
	provided that $D_z\II(\hat{t}(s),\hat{z}(s)) \neq 0$. The condition in \eqref{eq:conditionz'} 
	will in general not be fulfilled by the linear interpolant from \eqref{eq:defhatz}, 
	but, of course, one can find other Lipschitz continuous functions
	satisfying \eqref{eq:conditionz'} as well as $\hat z(s(t_n-)) = z(t_n-)$ and $\hat z(s(t_n)) = z(t_n)$. 
    However, in view of \eqref{eq:defhatell} and the required boundedness of $\hat\ell$ as a function of bounded variation, 
    $\hat z'$ must be bounded away from zero in $(s(t_n-), s(t_n))$. This is however not possible, if $\hat z'$ is smooth and 
    $-D_z\II(\hat{t}(s),\hat{z}(s))$ and $z(\hat{t}(s))-z(\hat{t}(s)-)$ have a different sign. If $\hat z'$ is not smooth, an additional 
    discontinuity in $\hat\ell$ arises, cf.\ \eqref{eq:defhatell}, contradicting again the required regularity of $\hat\ell$, if this happens 
    more than finitely many times with a jump height that does not tend to zero sufficiently fast. 
    All in all, we see that the construction of $\hat\ell$ in the presence of countably many jumps of $z$ is very involved and seems hardly to 
    be possible in general.
\end{remark}

\subsection{A second counterexample}\label{sec:extwo}

In order to further investigate the relation between the relaxed solution concept and the notion of local solutions, let 
us return to the counterexample from Section~\ref{sec:acounterexample} and slightly modify it. 
While energy and dissipation are left unchanged, see \eqref{eq:dissandenergy}, we now consider 
\begin{align}\label{eq:seqelln}
	\ell_n(t)=\begin{cases}
				0,& t\in[0,1],\\
				\frac{n}{2}t-\frac{n}{2}, &t\in(1,1+\frac{1}{n}),\\
				0, &t\in[1+\frac{1}{n},2]
			\end{cases}
\end{align}
as the sequence of loads. This sequence converges pointwise everywhere to $\ell \equiv 0$. 
Moreover, it converges weakly$*$ in $\BV(0,T)$ to zero, too, but of course not in the intermediate sense, since 
$|D\ell_n|(0,2) = 1 \neq 0 = |D\ell|(0,2)$ for all $n\in \N$.
The initial state is again set to $z_0 = 0$.  
By direct calculations one verifies that a relaxed solution associated with $\ell_n$ is again given by
\eqref{eq:hatzex} and \eqref{eq:hattex} along with $S_n = 5/2$ and 
\begin{equation}
    \hat{\ell}_n(s)=\ell_n(\hat{t}_n(s))=
    \begin{cases}
        0, &s\in[0,1],\\
        \frac{\frac{n}{2}}{1+\frac{n}{2}}s-\frac{\frac{n}{2}}{1+\frac{n}{2}}, &s\in (1,\frac{3}{2}+\frac{1}{n}),\\
        0, &s\in[\frac{3}{2}+\frac{1}{n},\frac{5}{2}]
    \end{cases}
\end{equation}
instead of \eqref{eq:hatellex}. Note that this is not only a relaxed solution, but also a normalized $\mathfrak{p}$-parametrized 
BV solution, as no jump occurs, i.e., $\hat t(s) > 0$ everywhere, and both concepts coincide in this case.

Again the limits of $\hat z$ and $\hat t$  (w.r.t.\ pointwise convergence as well as weak$*$ convergence in $W^{1,\infty}(0,S)$)  are given by the 
functions in \eqref{eq:limitzt}, but $\hat\ell_n$ now converges to 
\begin{equation}\label{eq:hatelllimit}
    \hat{\ell}(s) = 
    \begin{cases}
        0, &s\in[0,1],\\
        s-1, &s\in (1,\frac{3}{2}),\\
        0, &s\in[\frac{3}{2},\frac{5}{2}].
    \end{cases}
\end{equation}
We observe that $\hat\ell$ is the pointwise limit and the limit w.r.t.\ intermediate convergence.
Again, the limit is no normalized, $\mathfrak{p}$-parametrized BV solution associated with $\ell$, since 
$\hat\ell$ from \eqref{eq:hatelllimit} is not compatible with $\ell \equiv 0$ 
in the viscous jump $(1,\frac{3}{2})$ according to the condition in \eqref{eq:defell}.
However, as all other requirements in Definition~\ref{def:paramsol} are satisfied, i.e., \eqref{eq:initendcond}--\eqref{eq:energy}, 
and, since $\hat\ell = \ell\circ \hat t$ in $M = (0,\frac{5}{2})\setminus (1,\frac{3}{2})$, the limit is indeed 
a relaxed solution in accordance with our findings in Theorem~\ref{thm:stability}.
Note in this context that $\hat\ell$ and $\hat z$ satisfy \eqref{eq:defhatell} in the viscous jump $(1,\frac{3}{2})$.

Let us now consider the example in physical time. Here we observe that $z_n\in W^{1,\infty}(0,2)$, defined by 
\begin{align}\label{eq:seqzn}
	z_n(t)=\begin{cases}
				0,& t\in[0,1],\\
				\frac{n}{2}t-\frac{n}{2}, &t\in(1,1+\frac{1}{n}),\\
				\frac{1}{2}, &t\in[1+\frac{1}{n},2],
				\end{cases},
\end{align}    
satisfies \eqref{eq:ris} a.e.\ in $(0,2)$, where energy and dissipation are given by \eqref{eq:dissandenergy} and $\ell_n$ is as defined in \eqref{eq:seqelln}.
Thus, $z_n$ is even a \emph{differentiable solution} and consequently a local solution, too. 
Due to the uniform convexity of the energy in \eqref{eq:dissandenergy}, the differential solution is unique (if it exists). Therefore, the only 
differential solution associated with the limit $\ell \equiv 0$ is $z \equiv 0$. As the initial state $z_0 = 0$ is locally stable, this is the only physically 
meaningful solution in the absence of external loading. But the sequence in \eqref{eq:seqzn} converges pointwise to
\begin{equation}\label{eq:limitzn}
	\tilde z(t)=\begin{cases}
			0, &t\in[0,1]\\
			\frac{1}{2} &t\in(1,2].
		\end{cases}
\end{equation}
This function however is not even a local solution, since the energy inequality \eqref{eq:energyineqlocsol} is not fulfilled.
To see this, observe that, for every  $t_2>1$, we obtain
\begin{equation}\label{eq:wrongsign}
    \frac{1}{8} = \II(t_2,\tilde z(t_2))+\mbox{Diss}_\RR(\tilde z;[0,t_2]) 
    > \II(0, \tilde z(0)) - \int_0^{t_2} \tilde z(r) \,\d\ell(r) = 0.
\end{equation}
A graphic interpretation of this observation is as follows: Due to the external force $\ell_n$, 
the state $z_n$ is raised to a higher energy level. As $n$ is increased, this shifting is accelerated such that, in the limit, 
the state remains on the higher energy level although the load vanishes. 
As a consequence, this increase of energy is not compensated by $\ell$ and the energy inequality is violated. 
This example shows that not only the notion of normalized $\mathfrak{p}$-parametrized BV solutions, 
but even the concept of local solutions is not stable w.r.t.\ a sequence of loads converging weakly$*$ in $\BV(0,T)$.

Moreover, it can be shown that the differential solution from \eqref{eq:seqzn} is also a \emph{global energetic solution}.
For the precise definition of this solution concept see \cite[Definition~2.1.2]{MRRIS}.
Since every global energetic solution is automatically a local solution,  the above passage to the limit shows
that the concept of global energetic solutions is not stable w.r.t.\ weak$*$ convergence of the loads, too.
In summary, we see that the relaxed solution concept is the only one which is stable w.r.t.\ weak$*$ convergence of loads.

Let us finally investigate if there is another local solution that coincides with the limit $(\hat z, \hat t)$
in the sense that $\bar z \in \mathfrak{P}(\hat t, \hat z)$, cf.\ \eqref{eq:defP}.
For this purpose, we transform the limit $\hat{z}$ back into physical time, which yields
\begin{equation*}
    \bar z(t)\in\hat{z}(\hat{t}^{-1}(t)) =
    \begin{cases} 
        0, &t\in[0,1),\\
        [0,\frac{1}{2}], &t=1,\\
        \frac{1}{2},&t\in(1,2],
    \end{cases}
\end{equation*}
where $\hat{t}^{-1}$ is the set-valued inverse of $\hat{t}$. Thus, for every $\bar z \in \mathfrak{P}(\hat t, \hat z)$, the 
inequality in \eqref{eq:wrongsign} holds true, which shows that no element of $\mathfrak{P}(\hat t, \hat z)$ is a local solution.
Even if we consider a modified load corresponding to the transformation of $\hat\ell$ from \eqref{eq:hatelllimit} into 
physical time, i.e., 
\begin{equation*}
    \bar\ell(t)\in\hat{\ell}(\hat{t}^{-1}(t)) = 
    \begin{cases} 
        0, &t\neq 1.\\
        [0,\frac{1}{2}], &t=1,
    \end{cases}
\end{equation*} 
the inequality in \eqref{eq:wrongsign} will still be valid, since the Kurzweil integral on the right side of \eqref{eq:wrongsign} 
is always zero, no matter how one chooses $\bar\ell$ in $t=1$, see \cite[Corollary 2.14]{Tvrdy}. This shows that, 
in general, one cannot interpret a relaxed solution $(S,\hat{t},\hat{z},\hat{\ell})$ as a local solution.

The above example illustrates the following: the price one has to pay for a solution concept that is stable w.r.t.\ weak$*$ convergence 
of the loads is that ``solutions'' are allowed that are not meaningful from a physical perspective. 
As already mentioned, the only reasonable solution of \eqref{eq:ris} without external loads (i.e., $\ell \equiv 0$) and
with a uniformly convex energy and a locally stable initial state $z_0$ is $z \equiv z_0$. The relaxed solution however jumps 
at $t=1$ from $z_0 = 0$ to $1/2$ without any impact of an external load, which makes of course no sense at all.

\section{Conclusion}

The two examples from Sections~\ref{sec:acounterexample} and \ref{sec:extwo} illustrate that established solution concepts 
are not stable w.r.t.\ loads converging in $BV([0,T];\R^d)$ in the sense that the limit of an associated sequence of solutions to the 
rate independent system need not be a solution any more. With regard to normalized $\mathfrak{p}$-parametrized 
BV solutions this even occurs for loads converging in the intermediate sense, as the example in Section~\ref{sec:acounterexample} shows. 
Concerning local (and thus also global energetic) solutions, the example in Section~\ref{sec:extwo} shows the same lack of stability 
for loads converging weakly$*$ in $BV([0,T];\R^d)$. 
As the analysis in Section~\ref{sec:stability} shows, it is possible to design a relaxed solution concept that is stable w.r.t.\ weak$*$ convergence 
of the loads. However, the considerations in Section~\ref{sec:plausibility} indicate that such a concept is even weaker than the notion of 
local solutions and thus not really meaningful in practice. Especially the example in Section~\ref{sec:extwo} shows that completely 
unphysical limits may arise, if the loads just converge weakly$*$ in $BV([0,T];\R^d)$, even in case of a uniformly convex energy.
This already demonstrates that a stable solution concept must necessarily allow for physically unreasonable solutions.

On the other hand, stability w.r.t.\ weak$*$ convergence of the loads is an essential property of a solution concept, as one can often hardly expect 
more than this type of convergence. This issue concerns the existence of optimal solutions as discussed at the end of Section~\ref{sec:stability}, 
but will very likely also arise when a given load in $BV([0,T];\R^d)$ is discretized in time.
All in all, our results indicate that it might make no sense to consider discontinuous loads when dealing with 
rate independent systems due to a lack of stability of physically meaningful solution concepts.

%%%%%%%%%%%%%%%%%%%%%%%%%%%%%%%%%%%%%%%%%%%%%%%%%%%%%%%%%%%%%%%%%%%%%%%%% 
%%%%%%%%%%%%%%%%%%%%%%%%%%%%%%%%%%%%%%%%%%%%%%%%%%%%%%%%%%%%%%%%%%%%%%%%% 
%%%%%%%%%%%%%%%%%%%%%%%%%%%%%%%%%%%%%%%%%%%%%%%%%%%%%%%%%%%%%%%%%%%%%%%%% 
%%%%%%%%%%%%%%%%%%%%%%%%%%%%%%%%%%%%%%%%%%%%%%%%%%%%%%%%%%%%%%%%%%%%%%%%% 
%%%%%%%%%%%%%%%%%%%%%%%%%%%%%%%%%%%%%%%%%%%%%%%%%%%%%%%%%%%%%%%%%%%%%%%%% 

\begin{appendix}

\section{The dissipative distance in the smooth case}

\begin{lemma}\label{lem:DissR}
	Let $\RR:\R^d\to [0,\infty)$ comply with \eqref{R1}--\eqref{R3} and $z\in W^{1,1}(0,T;\R^d)$. Then 
	\begin{align}
		\textnormal{Diss}_\RR(z;[t_1,t_2])=\int_{t_1}^{t_2}\RR(z'(r))dr
	\end{align}
	for all $0\leq t_1\leq t_2\leq T$.
\end{lemma}

\begin{proof}
	Let an arbitrary partition $t_1=\xi_0<\xi_1<\dots<\xi_{n-1}<\xi_n=t_2$ be given. Then applying Jensen's inequality gives
	\begin{align*}
		\RR(z(\xi_i)-z(\xi_{i-1}))=\RR\Big(\int_{\xi_{i-1}}^{\xi_i}z'(r)dr\Big)\leq \int_{\xi_{i-1}}^{\xi_i}\RR(z'(r))\, dr
	\end{align*}
	for all $i=1,\dots,n$.
	Summing up results in $\sum_{i=1}^n \RR(z(\xi_i)-z(\xi_{i-1}))\leq \int_{t_1}^{t_2}\RR(z'(r))dr$ and, as the partition was arbitrary, this gives
	$\mbox{Diss}_\RR(z;[t_1,t_2])\leq\int_{t_1}^{t_2}\RR(z'(r))dr$.
	
	To show the reverse estimate, let a sequence of partitions $(\{\xi_i^h\}_{i=1}^{n_h})_h$ with $h:=\max_{i=1,\dots,n_h}\xi_i^h-\xi_{i-1}^h\to 0$ be given. 
	We define 
	\begin{equation*}
		\xi_h(t):=\frac{z(\xi_i^h)-z(\xi_{i-1}^h)}{\xi_i^h-\xi_{i-1}^h},\quad t\in[\xi_{i-1}^h,\xi_{i}^h).
	\end{equation*}
     Then the 1-homogeneity of $\RR$ yields
	\begin{align*}
		\mbox{Diss}_\RR(z;[t_1,t_2])&\geq\sum_{i=1,\dots,n_h}\RR(z(\xi_i^h)-z(\xi_{i-1}^h))\\
		&=\sum_{i=1,\dots,n_h}\int_{\xi_{i-1}^h}^{\xi_{i}^h}\RR\bigg(\frac{z(\xi_i^h)-z(\xi_{i-1}^h)}{\xi_i^h-\xi_{i-1}^h}\bigg) \, dr
		=\int_{t_1}^{t_2}\RR(\xi_h(r)) \,dr.
	\end{align*} 
	By Lebesgue's differentiation theorem we obtain $z'(t)=\lim_{h\searrow 0}\xi_h(t)$  
	for almost all $t\in(0,T)$ such that the lower semicontinuity of $\RR$ implies $\liminf_{h\searrow0}\RR(\xi_h(t))\geq \RR(z'(t))$ 
	a.e.\ in $(0,T)$.	Thus Fatou's lemma yields
	\begin{align*}
		\mbox{Diss}_\RR(z;[t_1,t_2]) &\geq \liminf_{h\searrow0}\int_{t_1}^{t_2}\RR(\xi_h(r))\,dr\\
		&\geq \int_{t_1}^{t_2}\liminf_{h\searrow0}\RR(\xi_h(r))\, dr\geq \int_{t_1}^{t_2} \RR(z'(r))\, dr,
	\end{align*}
	which finally gives the claim.
\end{proof}

\section{Auxiliary results from convex analysis}

The following results are well known, we only refer to \cite{michael} and the references therein. For convenience of the reader, we present the proofs in detail.

\begin{lemma}\label{lem:ygrene}
	For every $(S,\hat{t},\hat{z},\hat{\ell})\in (0,\infty) \times W^{1,\infty}(0,S)\times W^{1,\infty}(0,S;\R^d)\times BV([0,S];\R^d)$ 
    and all $0\leq s_1\leq s_2\leq S$ there holds
	\begin{equation}
	\begin{aligned}\label{eq:ygrene}
		\EE(\hat{z}(s_2))
		+\int_{s_1}^{s_2}\RR(\hat{z}'(r))+\|\hat{z}'(r)\|\abst(-D_z\hat{\II}(r,\hat{z}(r)),\partial\RR(0)) \, dr \qquad & \\
		\geq\EE(\hat z(s_1))+\int_{s_1}^{s_2}\langle \hat{\ell}(r),\hat{z}'(r)\rangle \, dr . &
	\end{aligned}
	\end{equation}
\end{lemma}

\begin{proof}
	The compactness of $\partial\RR(0)$ implies that,  for all $r\in(s_1,s_2)$, there exists $w_r\in\partial \RR(0)$ 
	such that $\abst(-D_z\hat{\II}(r,\hat{z}(r)),\partial\RR(0))=\|-D_z\hat{\II}(r,\hat{z}(r))-w_r\|$. From this we conclude
	\begin{align*}
		\langle -D_z\hat\II(r,\hat z(r)),\hat{z}'(r)\rangle
		& = \langle -D_z\hat\II(r,\hat z(r))-w_r,\hat{z}'(r)\rangle+\langle w_r,\hat{z}'(r)\rangle\\
		& \leq \| -D_z\hat\II(r,\hat z(r))-w_r\| \, \|\hat{z}'(r)\|+\langle w_r,\hat{z}'(r)\rangle\\
		& \leq \|\hat{z}'(r)\|\abst(-D_z\hat{\II}(r,\hat{z}(r)),\partial\RR(0))+\RR(\hat z'(r)),
	\end{align*}
	where we used $w_r \in \partial\RR(0)$ for the last inequality.
	Now integrating over $(s_1,s_2)$ yields
	\begin{align*}
		& \int_{s_1}^{s_2} \|\hat{z}'(r)\|\abst(-D_z\hat{\II}(r,\hat{z}(r)),\partial\RR(0))+\RR(\hat z'(r)) \, dr\\
		& \quad \geq \int_{s_1}^{s_2} \langle -D_z\hat\II(r,\hat z(r)),\hat{z}'(r)\rangle \, dr \\
		& \quad = \int_{s_1}^{s_2}\langle -D_z\EE(\hat z(r)),\hat z'(r)\rangle +\langle \hat{\ell}(r),\hat{z}'(r)\rangle \, dr\\
		& \quad = \EE(\hat z(s_1))-\EE(\hat z (s_2))+ \int_0^s \langle \hat{\ell}(r),\hat{z}'(r)\rangle \, dr,
	\end{align*}
	which is \eqref{eq:ygrene}.
\end{proof}

\begin{lemma}\label{lem:Rstar}
	Let $\RR:\R^d\to[0,\infty)$ comply with \eqref{R1}--\eqref{R3} and $\tau>0$ be given.
    Then there holds
    \begin{equation*}
        \tau \, \abst(\eta,\partial\RR(0)) = \sup_{v\in \R^d, \, \|v\|\leq \tau} \big( \dual{\eta}{v} - \RR(v) \big) 
    \end{equation*}
    for all $\eta \in \R^d$.
\end{lemma}

\begin{proof}
    Let us denote the indicator functional of a set $M \subset \R^d$ by $I_M$, i.e.,
    \begin{equation*}
        I_M(z) := \begin{cases} 0, &\text{if}~ z \in M, \\  \infty, &\text{else.} \end{cases}
    \end{equation*}        
    Then the inf-convolution formula from \cite[Prop. 3.4]{Attouch} yields
	\begin{equation}\label{eq:infconv}
		\big(\RR+I_{\overline{B_\tau(0)}}\big)^*(\eta)=\inf_{v\in\R^d} \Big(\RR^*(v)+I_{\overline{B_\tau(0)}}^*(\eta-v)\Big),
	\end{equation}
	where $B_\tau(0)$ denotes the ball of radius $\tau$ centered at zero.
	The conjugate of $I_{\overline{B_\tau(0)}}$ is given by $I_{\overline{B_\tau(0)}}^*(\cdot) = \tau\|\cdot\|$, while the positive homogeneity of $\RR$ implies 
    $\RR^* = I_{\partial\RR(0)}$.
	Inserted in \eqref{eq:infconv}, we obtain
	\begin{equation*}
		\big(\RR+I_{\overline{B_\tau(0)}}\big)^*
		=\inf_{v\in \R^d} I_{\partial\RR(0)}(v)+\tau\|\eta-v\|=\tau \inf_{v\in\partial\RR(0)}\|\eta-v\| =\tau \abst(\eta,\partial\RR(0)),
	\end{equation*}
	which is the claim.
\end{proof}

\end{appendix}

\bibliographystyle{acm} \bibliography{literatur}

\end{document}